\documentclass[11pt,a4paper]{article}
\usepackage[utf8]{inputenc}
\usepackage{amsmath,amssymb,amsthm,amsfonts,latexsym,graphicx,subfigure,enumerate,tikz}

\usepackage{xcolor}

\theoremstyle{plain}
\newtheorem{theo}{Theorem}
\newtheorem{lemma}[theo]{Lemma}
\newtheorem{prop}[theo]{Proposition}

\theoremstyle{definition}

\newtheorem{definition}{Definition}
\newtheorem{rmk}{Remark}

\DeclareMathOperator{\id}{Id}
\usepackage{hyperref}
\usepackage{comment}

\usepackage[affil-it]{authblk}

\title{Petals and Books:\\ The largest Laplacian spectral gap  from $1$}
\author[1,2]{Jürgen Jost}
\author[3,1]{Raffaella Mulas\footnote{Email address: raffaella.mulas@mis.mpg.de}}
\author[4,1]{Dong Zhang}
\affil[1]{Max Planck Institute for Mathematics in the Sciences, Leipzig, Germany}
\affil[2]{Santa Fe Institute for the Sciences of Complexity, Santa Fe, New Mexico, USA}
\affil[3]{Vrije Universiteit Amsterdam, Amsterdam, The Netherlands}
\affil[4]{LMAM and School of Mathematical Sciences, Peking University, Beijing, China}

\date{}

\begin{document}
\bibliographystyle{plain} 

\maketitle

\begin{abstract}
We prove that, for any connected  graph on $N\geq 3$ vertices, the spectral gap from the value $1$ with respect to the normalized Laplacian is at most $1/2$. Moreover, we show that equality is achieved if and only if the graph is either a petal graph (for $N$ odd) or a book graph (for $N$ even). This implies that $(\frac12,\frac32)$ is a maximal gap interval for the   normalized   Laplacian on connected graphs. This is  closely related to the  Alon-Boppana bound on regular graphs   and a recent result by Koll\'ar and Sarnak on cubic graphs. Our result also provides a sharp bound for the convergence  rate of some eigenvalues of  the Laplacian on neighborhood graphs.
    \vspace{0.2cm}

\noindent {\bf Keywords:} Spectral graph theory; Normalized Laplacian; Spectral gaps; Eigenvalue 1; maximal gap interval; neighborhood graph
\end{abstract}

\section{Introduction}
A spectral gap is the maximal difference between two eigenvalues for linear operators in some class. Here, we consider Laplace operators, more precisely, the normalized Laplacian of a finite graph. Such inequalities were first studied for the Laplace operator $\Delta=-\sum_{i=1}^n\frac{\partial^2 }{\partial (x^i)^2}$ on a connected smooth domain $\Omega \subset \mathbb{R}^n$ with Dirichlet conditions, that is,
\begin{eqnarray}
\label{1}
\Delta u &=&\lambda u \text{ in }\Omega\\
\label{2a}
u&=&0 \text{ on }\partial \Omega
\end{eqnarray}
where a (smooth) solution $u$ that does not vanish identically is called an eigenfunction, and the corresponding value $\lambda$ an eigenvalue. We choose the sign in the definition of $\Delta$ to make it a non-negative operator. Since $\Omega$ is connected, the smallest eigenvalue $\lambda_1$ is positive, and the famous Faber-Krahn inequality \cite{Faber23,Krahn25,Krahn26} says that among all domains with the same volume, the smallest possible value is realized by a ball of that volume. By a result of Ashbaugh and Benguria \cite{Ashbaugh92}, the ratio between the first two eigenvalues, $\frac{\lambda_2}{\lambda_1}$ is largest for the ball. In either case, equality is assumed precisely for the ball. When we replace the Dirichlet boundary condition \eqref{2a} by the Neumann boundary condition
\begin{equation}\label{2b}
 \frac{\partial u}{\partial n} =  0 \text{ on }\partial \Omega,
\end{equation}
then the smallest eigenvalue is $\lambda_1=0$, with the constants being eigenfunctions, and Weinberger  \cite{Weinberger56} proved that the second smallest eigenvalue $\lambda_2$ now is always less than or equal to that of the ball, again with equality only for the ball. While the proofs of such results can be difficult, there is a general pattern here, that the extremal cases occur only for a very particular class of domains, balls in this case. Similarly, for eigenvalue problems for the Laplace-Beltrami operator in Riemannian geometry, often the extremal case is realized by spheres (for the eigenvalue problem in Riemannian geometry, see for instance the references given in  \cite{Chavel84,Jost17}). 

There are also discrete versions of those Laplacians, and naturally, their spectra have also been investigated. For the algebraic graph Laplacian, a systematic analysis of spectral gaps is presented in \cite{Kollar20}, and these authors have identified many beautiful classes of graphs with a particular structure of their spectra and spectral gaps.
Here, we consider another discrete  Laplacian, the normalized Laplace operator of a connected, finite, simple graph $\Gamma=(V,E)$ on $N\geq 3$ vertices. For a vertex $v\in V$, we denote by $\deg v$ its degree, that is, the number of its neighbors, i.e., the other vertices $w\sim v$ connected to $v$ by an edge. Then the Laplacian for a function $f:V\to \mathbb{R}$ is
\begin{equation}\label{eq:defL}
    \Delta f(v)=f(v)-\frac{1}{\deg v} \sum_{w\sim v}f(w),
\end{equation}
that is, we subtract from the value of $f$ at $v$ the average of the values at its neighbors. This operator generates random walks and diffusion processes on graphs, and it 
was first systematically studied in \cite{Chung97}. Since its spectrum is that of an $(N\times N)$-matrix with 1s in the diagonal, the eigenvalues $\lambda_1 \le \lambda_2 \le \dots \le \lambda_N$ satisfy
\begin{equation}\label{3}
    \sum_{i=1}^N \lambda_i =N.
\end{equation}
The first eigenvalue now is $\lambda_1=0$, but since $\Gamma$ is connected, the second eigenvalue $\lambda_2$ is positive. There are several inequalities controlling 
$\lambda_2$ from below in terms of properties of the graph (see \cite{Chung97}), and the largest value among all graphs with $N$ vertices is realized by the complete graph $K_N$ where $\lambda_2=\frac{N}{N-1}$; for all other graphs $\lambda_2\le 1$ \cite[Lemma 1.7]{Chung97}. The largest eigenvalue $\lambda_N$ is always less than or equal to $2$, with equality if and only if $\Gamma$ is bipartite. And the gap $2-\lambda_N$ quantifies how different $\Gamma$ is from being bipartite \cite{Bauer13}. In fact, the smallest possible value $\lambda_N=\frac{N}{N-1}$ is again realized only for $K_N$. For all other graphs, $\lambda_N\ge \frac{N+1}{N-1}$ \cite{Das16}, and again, the extremal graphs, where equality is realized, can be characterized \cite{JMM}. 

Thus, the situation for the spectral gaps at the ends of the spectrum, that is, at 0 and near 2 has been clarified. But we may also ask about gaps in the middle of the spectrum. In fact, besides 0 and 2, also the eigenvalue 1 plays a special role. It arises, in particular, from vertex duplications \cite{Banerjee08}. The extreme case is given by complete bipartite graphs $K_{n,m}$ with $n+m=N$. They have the eigenvalue 1 with multiplicity $N-2$. In fact, if 1 occurs with this multiplicity, then by \eqref{3}, 2 also has to be an eigenvalue, and the graph is bipartite. But like the eigenvalue 2, the eigenvalue 1 need not be present in a graph. Therefore, we can ask about the maximal spectral gap at 1, that is, we can ask what the maximal value of 
\begin{equation*}
    \varepsilon:=\min_i |1-\lambda_i|
\end{equation*}
could be. In this paper, we show that for any graph with $N\ge 3 $ vertices,  $\varepsilon \le \frac{1}{2}$ (for the graph with two vertices, the eigenvalues are 0 and 2, therefore, in this particular case, the gap is $1$), and as the title already reveals, we can identify the class of graphs for which the maximal possible value $\varepsilon = \frac{1}{2}$ is realized. In fact, in those cases, except for the triangle $K_3$, which only has $\frac{3}{2}$, both values $\frac{1}{2}$ and $\frac{3}{2}$ are eigenvalues. 

Our results have fit into a larger picture. They have  connections with  expander graphs and random walks on graphs, including  Alon-Boppana's theorem on Ramanujan graphs, Koll\'ar-Sarnak's theorem on the  maximal gap interval for cubic graphs \cite{Kollar20}, as well as Bauer-Jost's Laplacian on neighborhood  graphs \cite{Bauer13}. We shall explain these relations in Section \ref{sec:main}.
\section{Main result}\label{sec:main}
Throughout the paper we fix a connected, finite, simple graph $\Gamma=(V,E)$ on $N\geq 3$ vertices. We let $d$ denote the smallest vertex degree, we let $C(V)$ denote the space of functions $f:V\rightarrow \mathbb{R}$ and, given a vertex $v$, we let $\mathcal{N}(v):=\{w\in V:w\sim v\}$ denote the neighborhood of $v$. We let
\begin{equation*}
    \lambda_1=0<\lambda_2\leq\ldots\leq\lambda_N
\end{equation*}
denote the eigenvalues of the Laplacian in \eqref{eq:defL}, and we let 
\begin{equation*}
    \varepsilon:=\min_i |1-\lambda_i|
\end{equation*}be the spectral gap from $1$. Clearly, $\varepsilon\geq 0$ and this inequality is sharp since, for instance, any graph with duplicate vertices (i.e., vertices that have the same neighbours \cite{medina2006}) has $1$ as an eigenvalue. As anticipated in the introduction, here we prove that $\varepsilon\leq \frac{1}{2}$ and equality is achieved if and only if $\Gamma$ belongs to one of the following two optimal classes.

\begin{definition}[Petal graph, $N\geq 3$ odd]
Given $m\geq 1$, the $m$--\emph{petal graph} is the graph on $N=2m+1$ vertices such that (Figure \ref{fig:petal}):
\begin{itemize}
    \item $V=\{x,v_1,\ldots,v_m,w_1,\ldots,w_m\}$;
    \item $E=\{(x,v_i)\}_{i=1}^m\cup \{(x,w_i)\}_{i=1}^m\cup\{(v_i,w_i)\}_{i=1}^m$.
\end{itemize}
\end{definition}

\begin{figure}[h]
    \centering
    \includegraphics[width=5cm]{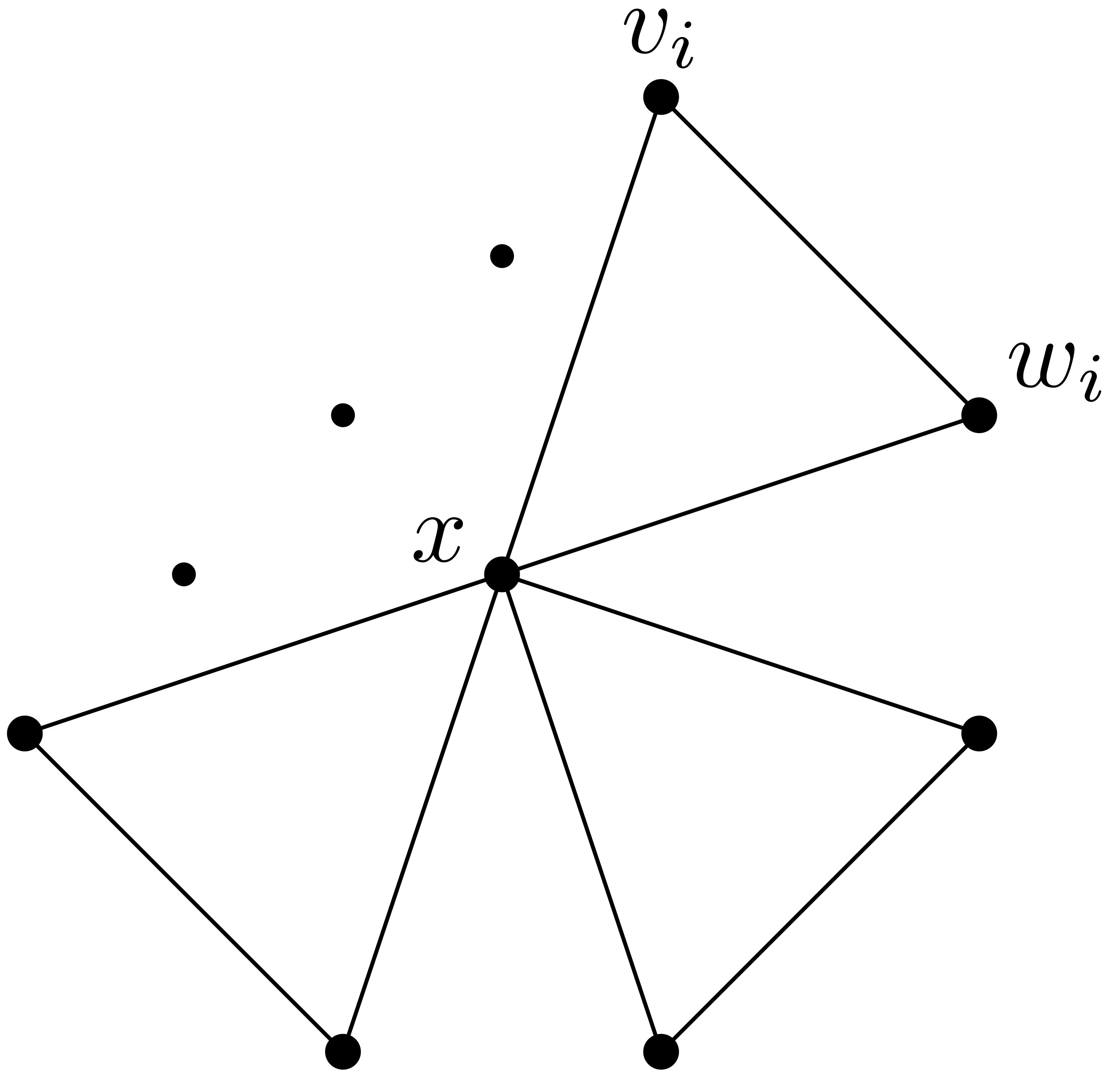}
    \caption{The petal graph}
    \label{fig:petal}
\end{figure}

 Petal graphs are also known as \emph{Dutch windmill graphs} or \emph{friendship graphs}. They appear in the famous Friendship Theorem from Erd\"os, R\'enyi and S\'os \cite{Erdos66}, which states that the only finite graphs with the property that every two vertices have exactly one neighbor in common are precisely the petal graphs. In fact, a proof of this result can proceed via spectral methods. The friendship assumption determines the square of the adjacency matrix, and hence the spectrum of that matrix, and one then proceeds by showing that this spectrum implies that the graph in question has to be petal graph. This is another example, and one relevant for the present paper of how the structure of a graph can be determined from its spectrum.\newline
As shown in \cite{Chung97}, for the petal graph on $N=2m +1$ vertices, the eigenvalues are $0$, $\frac{1}{2}$ (with multiplicity $m-1$) and $\frac{3}{2}$ (with multiplicity $m+1$). Therefore, $\varepsilon=\frac{1}{2}$ in this case.\\
 In fact, the eigenfunctions are easily constructed. Putting $f(w_i)=-f(v_i)$ and $f(x)=0$ produces $m$ linearly independent eigenfunctions for the eigenvalue $\frac{3}{2}$. Letting  $f(v_i)=f(w_i)=1 $ for all $i$ and $f(x)=-2$  produces another eigenfunction for that eigenvalue. With $f(v_i)=f(w_i)$ for all $i$, $\sum_i f(v_i)=0=f(x)$ we get  $m-1 $  linearly independent eigenfunctions for the eigenvalue $\frac{1}{2}$. Again, this is a good example of how the structure of a graph and properties of its spectrum are tightly related. Whenever we have two neighboring vertices $v,w$ that have all their other neighbors in common, the function with $f(v)=1, f(w)=-1, f(z)=0$ for all other vertices is an eigenfunction, and the eigenvalue depends on the number of those common neighbors. When, as here, there is only one common neighbor for such a pair, the eigenvalue is $3/2$. This eigenvalue will also occur for the book graph to be introduced in a moment, which for an even number of pages is a two-fold cover of the petal graph. But the book graph will be bipartite, and hence have the eigenvalue 2, which the petal graph, not being bipartite, cannot possess. 
 
\begin{definition}[Book graph, $N\geq 4$ even]
Given $m\geq 1$, the $m$--\emph{book graph} is the graph on $N=2m+2$ vertices such that (Figure \ref{fig:book}):
\begin{itemize}
    \item $V=\{x,y,v_1,\ldots,v_m,w_1,\ldots,w_m\}$;
    \item $E=\{(x,v_i)\}_{i=1}^m\cup \{(y,w_i)\}_{i=1}^m\cup\{(v_i,w_i)\}_{i=1}^m$.
\end{itemize}
\end{definition}

\begin{figure}[h]
    \centering
    \includegraphics[width=7cm]{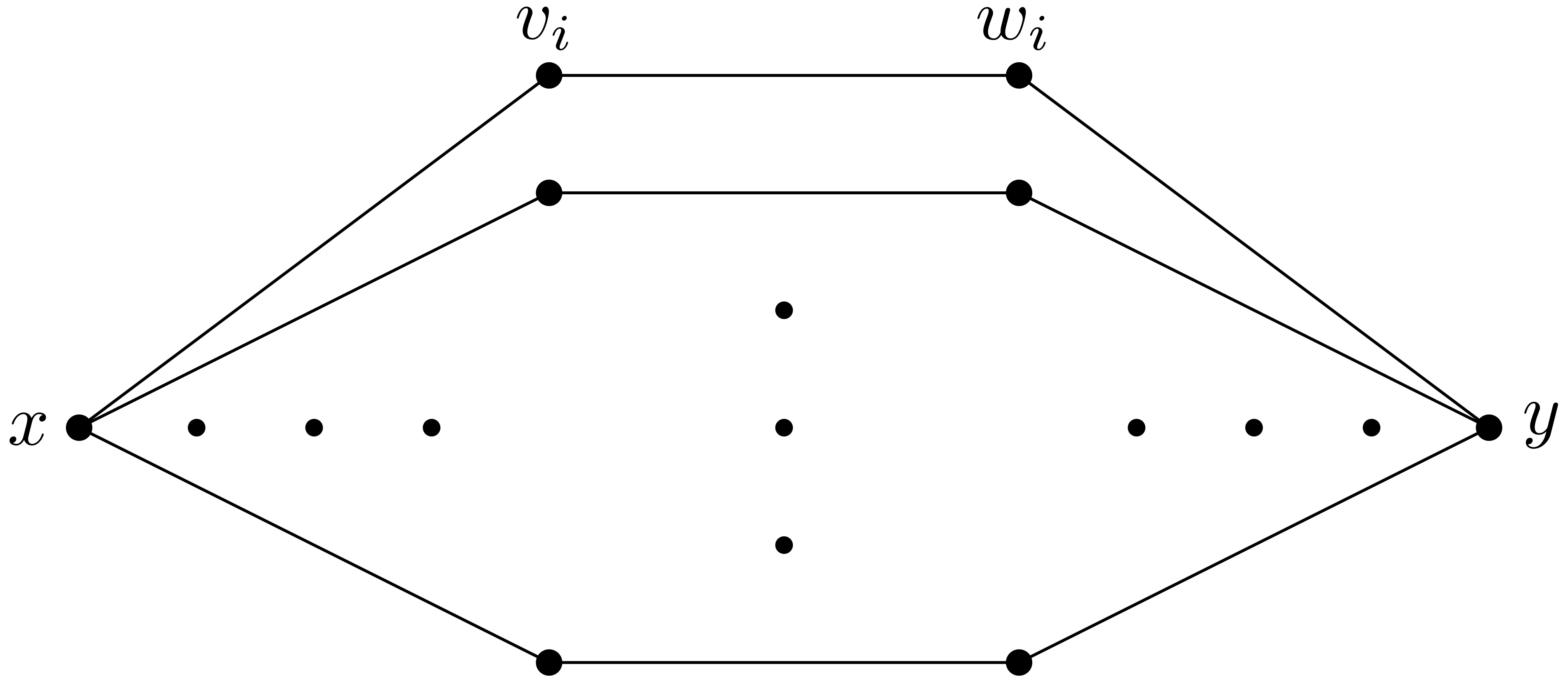}
    \caption{The book graph}
    \label{fig:book}
\end{figure}

\begin{rmk}
For the book graph on $N=2m+2$ vertices, $0$ and $2$ are eigenvalues with multiplicity $1$, since $\Gamma$ is connected and bipartite. Moreover, one can check that $\lambda=1\pm \frac{1}{2}$ are eigenvalues with multiplicity $m$ each. In fact, the corresponding eigenfunctions can be constructed as follows:
\begin{enumerate}
    \item By letting $\sum_i f(v_i)=0$, $f(x)=f(y)=0$, $f(w_i)=\mp f(v_i)$,  we obtain $m-1$ linearly independent eigenfunctions;
    \item By letting $f(v_i)=- f(w_i)=1$ and $f(x)=-f(y)=2$, we obtain one more eigenfunction for $\frac{1}{2}$ and similarly by letting $g(v_i)=g(w_i)=-1$ and $g(x)=g(y)=2$, we obtain one more eigenfunction for $\frac{3}{2}$. 
\end{enumerate}Hence, also in this case, $\varepsilon=\frac{1}{2}$.
\end{rmk}

Our main result is that these, and only these, examples have the largest possible spectral gap $\varepsilon =\frac{1}{2}$ at 1.

\begin{theo}\label{thm:main}
For any connected graph $\Gamma$ on $N\geq 3$ vertices,
\begin{equation*}
    \varepsilon\leq \frac{1}{2}.
\end{equation*} Moreover, equality is achieved if and only if $\Gamma$ is either a petal graph (for $N$ odd) or a book graph (for $N$ even).
\end{theo}

As a consequence of Theorem \ref{thm:main}, we can infer that both petal  graphs and book graphs are uniquely characterized by their normalized Laplacian spectra. This relates to \cite{Cioaba15}, where it has been proved that, among connected graphs, the petal graphs are uniquely determined by the eigenvalues of the adjacency matrix.\newline

We prove Theorem \ref{thm:main} in Section \ref{section:proof}. For completeness, we also state the following result on the value $\max_{i\neq 1}|\lambda_i-1|.$

\begin{prop}\label{pro:max-lamda-i}For any connected graph $\Gamma$ on $N\geq 2$ vertices, 
$$\frac{1}{N-1}\le\max\limits_{i\ne 1}|\lambda_i-1|\le1 .$$ 
Moreover, the lower bound is an equality if and only if $\Gamma$ is the complete graph; the upper bound is an equality if and only if $\Gamma$ is bipartite.
\end{prop}
\begin{proof}
By \eqref{3}, using the fact that $\lambda_1=0$, it follows that $\sum_{i=2}^N(\lambda_i-1)=1$. Thus,  
\begin{equation*}
    \frac{1}{N-1}\leq \frac{1}{N-1}\cdot \left(\sum_{i=2}^N|\lambda_i-1|\right)\leq \max\limits_{i\ne 1}|\lambda_i-1|
\end{equation*}and equality holds if and only if $\lambda_2=\ldots=\lambda_N=\frac{N}{N-1}$, that is, if and only if $\Gamma$ is the complete graph.\newline
The other claim follows from the fact that
\begin{equation*}
    0<\lambda_2\leq \ldots \leq \lambda_N\leq 2
\end{equation*}and $\lambda_N=2$ if and only if $\Gamma$ is bipartite.
\end{proof}
We also prove the following lemma, which will be needed in the proof of Theorem \ref{thm:main} and which is an interesting result itself, since it allows us to characterize $\varepsilon$ for any graph.

\begin{lemma}\label{lemma:first}  
For any graph $\Gamma$, 
$$\varepsilon^2=\min_{f\in C(V)\setminus\{ \mathbf{0}\}}\frac{\sum\limits_{w\in V}\frac{1}{\deg w}\left(\sum\limits_{v\in \mathcal{N}(w)}f(v)\right)^2}{\sum\limits_{w\in V} \deg w\cdot f(w)^2}.
$$
\end{lemma}
\begin{proof}

We observe that the values $(1-\lambda_1)^2,\ldots,(1-\lambda_N)^2$ are exactly the eigenvalues of the matrix $M:=(\id-D^{\frac12}\Delta D^{-\frac12})^2$ whose entries are 
$$M_{uv}=\sum_{w\in \mathcal{N}(u)\cap \mathcal{N}(v)}\frac{1}{\deg w\sqrt{\deg u\cdot \deg v}},\quad\text{ for }u,v\in V,$$
where $\id$ is the $N\times N$ identity matrix and $D=\mathrm{diag}(\deg 1,\cdots,\deg N)$ is the diagonal matrix consisting of degrees.  
In particular, $\varepsilon^2$ is the smallest eigenvalue of $M$. Therefore, by the Courant--Fischer--Weyl min-max principle, it can be written as
$$\varepsilon^2=\min\limits_{f\in C(V)\setminus\{ \mathbf{0}\}}\frac{\sum\limits_{u,v\in V}\sum\limits_{w\in \mathcal{N}(u)\cap \mathcal{N}(v)}\frac{1}{\deg w\sqrt{\deg u\cdot \deg v}}f(u)\cdot f(v)}{\sum\limits_{w\in V} f(w)^2}.$$
Now, observe that the numerator can be rewritten as
\begin{align*}
\sum\limits_{u,v\in V}\sum\limits_{w\in \mathcal{N}(u)\cap \mathcal{N}(v)}\frac{f(u)\cdot f(v)}{\deg w\sqrt{\deg u\cdot \deg v}}&=\sum\limits_{w\in V}\sum\limits_{u,v\in \mathcal{N}(w)}\frac{f(u)\cdot f(v)}{\deg w\sqrt{\deg u\cdot \deg v}}    
\\&=\sum\limits_{w\in V}\frac{1}{\deg w}\left(\sum\limits_{v\in \mathcal{N}(w)}\frac{f(v)}{\sqrt{\deg v}}  \right)^2.
\end{align*}
It follows that
\begin{align*}
\varepsilon^2&=\min\limits_{f\in C(V)\setminus\{ \mathbf{0}\}}\frac{\sum\limits_{w\in V}\frac{1}{\deg w}\left(\sum\limits_{v\in \mathcal{N}(w)}\frac{f(v)}{\sqrt{\deg v}}  \right)^2}{\sum\limits_{w\in V} f(w)^2}\\
&=\min_{f\in C(V)\setminus\{ \mathbf{0}\}}\frac{\sum\limits_{w\in V}\frac{1}{\deg w}\left(\sum\limits_{v\in \mathcal{N}(w)}f(v)\right)^2}{\sum\limits_{w\in V} \deg w\cdot f(w)^2}.
\end{align*}
\end{proof}
This lemma will play an important role for controlling $\varepsilon$, because we can derive inequalities on vertex degrees in case $\varepsilon\ge \frac{1}{2}$, by choosing suitable local functions $f$.
\begin{rmk}
A \emph{gap interval with respect to a family}  $\mathcal{G}$ of simple graphs    is an open interval such that there are infinitely many graphs in $\mathcal{G}$ whose Laplacian spectrum  does not intersect the interval. Clearly, Theorem \ref{thm:main} trivially implies that, for $\mathcal{G}=\{\text{connected 
graphs}\}$,  $(\frac12,\frac32)$ is a maximal gap interval. This can be compared to a result by Koll\'ar and Sarnak \cite{Kollar20}, which states that, for  $\mathcal{G}=\{\text{connected regular  graphs of degree }3\}$, $(\frac23,\frac43)$ is a maximal gap interval. (The original result, Theorem 3 in \cite{Kollar20}, is formulated in terms of the adjacency matrix, but since these graphs are regular, the statement can be equivalently reformulated in terms of $\Delta$.) We also refer to the recent work of the application of gap intervals to microwave coplanar waveguide resonators by  Koll\'ar et al.\ \cite{Kollar19}. This line of research has an origin in the  Alon-Boppana theorem    \cite{Nilli91,Chiu92},  which implies that   $(0,1-\frac{2\sqrt{2}}{3})$ is a maximal gap interval for the Laplacian on cubic graphs, and this gap is achieved by  Ramanujan graphs.
\end{rmk}
Combining Theorems \ref{thm:main} and \ref{thm:d3}, we have the following

\begin{theo}\label{thm:main2}
For any connected graph $\Gamma$ on $N\geq 3$ vertices with smallest degree $d\geq 2$,
   \begin{equation*}
        \varepsilon\leq\frac{\sqrt{d-1}}{d}.
    \end{equation*}
\end{theo}
This is rather interesting, and can be compared to   Alon–Boppana's work stating  that for any $d$-regular  
Ramanujan graph,  $$\max\limits_{i\ne 1}|\lambda_i-1| \leq 2\frac{\sqrt{d-1}}{d}.$$
See Proposition \ref{pro:max-lamda-i} for another comparison. 

Furthermore, our results provide a sharp bound for the convergence  rate of some eigenvalues of  the Laplacian on neighborhood graphs \cite{Bauer13}. 

The neighborhood graph   $\Gamma^{[\ell]}$ of order $\ell$ of  a graph  $\Gamma=(V,E)$ is a weighted graph whose edge weight $w_{ij}$  equals the probability that a random walker starting at $i$
reaches $j$ in $\ell$ steps. The neighborhood graphs $\{\Gamma^{[\ell]}\}_{\ell=1}^\infty$ encode
properties of random walks on $\Gamma$, asymptotic ones if $\ell\to \infty$. We
thereby gain a new source of geometric intuition for obtaining eigenvalue
estimates. The graph Laplacian $\Delta^{[\ell]}$  on
  $\Gamma^{[\ell]}$ satisfies
  $\Delta^{[\ell]}u= (I-(I-\Delta)^{l})u
  $
  for  $u\in\ell^2(\Gamma)$ (see \cite{Bauer13}). By Theorem  \ref{thm:main}, we have:
\begin{theo}\label{theonei1}
  For every connected graph  $\Gamma$ with at least three vertices, there is some eigenvalue
  $\lambda^{[\ell]}$ of $\Delta^{[\ell]}$ with
  \begin{equation*}
    \label{nei8}
  |1-  \lambda^{[\ell]}|\le \frac{1}{2^\ell}.
\end{equation*}
When $\ell$ is even, the largest eigenvalue of   $\Gamma^{[\ell]}$ satisfies
  \begin{equation*}
    \label{nei9}
    1- \frac{1}{2^\ell}\le \lambda^{[\ell]}_N\le 1,
  \end{equation*}
  and both bounds are sharp.
\end{theo}

\section{Proof of the main result}\label{section:proof}

This section is dedicated to the proof of Theorem \ref{thm:main}, that we split into two main parts. In particular, in Section \ref{section:upper} we prove that $\varepsilon\leq\frac{1}{2}$ for any connected graph $\Gamma$ on $N\geq 3$ vertices, while in Section \ref{section:optimal} we prove that $\varepsilon=\frac{1}{2}$ if and only if $\Gamma$ is either a petal graph or a book graph.

\subsection{Upper bound}\label{section:upper}
In this section we prove the first claim of Theorem \ref{thm:main}, namely 
\begin{theo}\label{thm:part1}
For any connected graph $\Gamma$ on $N\geq 3$ vertices,
\begin{equation*}
    \varepsilon\leq \frac{1}{2}.
\end{equation*} 
\end{theo}
Before proving it by contradiction, we show several properties that a connected graph on $N\geq 3$ vertices with $\varepsilon>\frac{1}{2}$ would have. We begin by observing that, by Lemma \ref{lemma:first}, such a graph should satisfy
\begin{equation}\label{eq:important}
\sum\limits_{w\in V}\frac{1}{\deg w}\left(\sum\limits_{v\in \mathcal{N}(w)}f(v)\right)^2>\frac14\sum\limits_{w\in V} \deg w\cdot f(w)^2, \quad \forall f\in C(V)\setminus \{\mathbf{0}\}.
\end{equation}

\begin{lemma}\label{lemma:deg3}Let $\Gamma$ be a connected graph on $N\geq 3$ vertices such that $\varepsilon>\frac{1}{2}$. Then, for any $v\in V$, there exists $w\in \mathcal{N}(v)$ such that $\deg w\le 3$. 
\end{lemma}

\begin{proof}
Let $f$ be such that $f(v):=1$ and $f(u):=0$ for all $u\ne v$. Then, \eqref{eq:important} implies $$\sum\limits_{w\in \mathcal{N}(v)}\frac{1}{\deg w}>\frac14 \deg v.$$
Now, if $\deg w\ge 4$ for all $w\in \mathcal{N}(v)$, then $$\sum\limits_{w\in \mathcal{N}(v)}\frac{1}{\deg w}\le \sum\limits_{w\in \mathcal{N}(v)}\frac14 =\frac14\deg v,$$
which is a contradiction.
\end{proof}

Given two vertices $u$ and $v$, we denote by $\mathcal{N}(u)\triangle \mathcal{N}(v)$ the symmetric difference of $\mathcal{N}(u)$ and $\mathcal{N}(v)$, i.e., the set of vertices that are neighbors of either $u$ or $v$.

\begin{lemma}\label{lemma:key-tech}Let $\Gamma$ be a connected graph on $N\geq 3$ vertices such that $\varepsilon>\frac{1}{2}$. If $\mathcal{N}(u)\cap \mathcal{N}(v)\ne\emptyset$, then $\mathcal{N}(u)\triangle  \mathcal{N}(v)\ne\emptyset$ and 
\begin{equation}\label{eq:repeat-key}
\sum_{w\in \mathcal{N}(u)\triangle \mathcal{N}(v)}\left(\frac{1}{\deg w}-\frac14\right)> \frac12.
\end{equation}
\end{lemma}

\begin{proof}
Let $f$ be such that $f(v):=1$, $f(u):=-1$ and $f(w):=0$ for all $w\in V\setminus\{u,v\}$. Then, by \eqref{eq:important},
 $$\sum_{w\in \mathcal{N}(u)\triangle \mathcal{N}(v)}\frac{1}{\deg w}>\frac14(\deg v+\deg u),$$
which also implies $\mathcal{N}(u)\triangle \mathcal{N}(v)\ne\emptyset$. Moreover, since
\begin{equation*}
    \deg v+\deg u=|\mathcal{N}(u)\triangle \mathcal{N}(v)|+2|\mathcal{N}(u)\cap \mathcal{N}(v)|,
\end{equation*}
the above inequality can be rewritten as
 $$\sum_{w\in \mathcal{N}(u)\triangle \mathcal{N}(v)}\left(\frac{1}{\deg w}-\frac14\right)> \frac14\cdot2|\mathcal{N}(u)\cap \mathcal{N}(v)|.$$

Since $\mathcal{N}(u)\cap \mathcal{N}(v)\ne\emptyset$, we have that $|\mathcal{N}(u)\cap \mathcal{N}(v)|\ge 1$, implying that
\begin{equation*}
\sum_{w\in \mathcal{N}(u)\triangle \mathcal{N}(v)}\left(\frac{1}{\deg w}-\frac14\right)> \frac12.
\end{equation*}
\end{proof}

\begin{lemma}\label{lemma:easycase}Let $\Gamma$ be a connected graph on $N\geq 3$ vertices such that $\varepsilon>\frac{1}{2}$. 
\begin{enumerate}
    \item If $\deg u=\deg w=1$, $u\neq w$, then there is no vertex $v$ with $u\sim v\sim w$.
    \item There are no distinct vertices  $u\sim v\sim w$ with
    \begin{equation*}
        1\le \deg u\le 2,\,\, 1\le \deg v\le 2 \text{ and }1\le \deg w\le 2.
    \end{equation*}
\end{enumerate} 
\end{lemma}

\begin{proof}
The first statement is a direct consequence of Lemma \ref{lemma:key-tech}.\newline 
We now prove the second statement. Assume the contrary, that is, there exist three vertices $u\sim v\sim w$ with $\{\deg u,\deg v,\deg w\}\subseteq\{1,2\}$. We only need to check the following two cases.
\begin{enumerate}[{Case} 1:] \item $\deg u=1$, $\deg v=2$ and $\deg w=2$.

Since $\mathcal{N}(u)\cap \mathcal{N}(w)\ne\emptyset$, by \eqref{eq:repeat-key}, the vertex $w_1\in \mathcal{N}(w)\setminus\{v\}$ has degree $1$. Hence, $\Gamma$ is a path of length $3$ (Figure \ref{fig:3path}). But this is a contradiction, since $\varepsilon=\frac{1}{2}$ for the path of length $3$, which is also the $1$-book graph.

\begin{figure}[h]
    \centering
    \begin{tikzpicture}[scale=1.5]
\draw (0,0)--(1,0)--(2,0)--(3,0);
\node (u) at  (0,0) {$\bullet$};
\node (v) at  (1,0) {$\bullet$};
\node (w) at  (2,0) {$\bullet$};
\node (w1) at  (3,0) {$\bullet$};
\node (u) at  (-0.1,0.2) {$u$};
\node (v) at  (0.8,0.2) {$v$};
\node (w) at  (2,0.2) {$w$};
\node (w1) at  (2.9,0.2) {$w_1$};
\end{tikzpicture}
    \caption{A graph arising in the proof of Lemma \ref{lemma:easycase}}
    \label{fig:3path}
\end{figure}
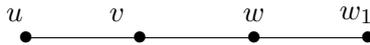
\item $\deg u=2$, $\deg v=2$ and $\deg w=2$.

Since $\mathcal{N}(u)\cap \mathcal{N}(w)\ne\emptyset$, by \eqref{eq:repeat-key}, the vertices $w_1\in \mathcal{N}(w)\setminus\{v\}$ and $u_1\in \mathcal{N}(u)\setminus\{v\}$ have degree $1$. Hence, $\Gamma$ is a path of length $4$. But this is also a contradiction, since $\varepsilon<\frac{1}{2}$ in this case.

\end{enumerate}
\end{proof}

\begin{lemma}\label{lemma:elementary-1}Let $\Gamma$ be a connected graph on $N\geq 3$ vertices such that $\varepsilon>\frac{1}{2}$. Then, there are no vertices  $u\sim v\sim w$ with $$2\le \deg u\le 3,\,\, 2\le \deg v\le 3 \text{ and } 2\le \deg w\le 3.$$
\end{lemma}

\begin{proof}Assume the contrary, that is, there exist three vertices $u\sim v\sim w$ with $\{\deg u,\deg v,\deg w\}\subseteq\{2,3\}$. We need to check five cases.

\begin{enumerate}[{Case} 1:]

\item $\deg u=2$, $\deg v=3$ and $\deg w=2$.

By applying \eqref{eq:repeat-key} three times, we have that:
\begin{itemize}
    \item There exists $w_1\in \mathcal{N}(u)\triangle \mathcal{N}(w)$ with $\deg w_1=1$, and without loss of generality, we may assume that $w_1\sim w$;
    \item There exists $v_1\in \mathcal{N}(v)\triangle \mathcal{N}(w_1)=\mathcal{N}(v)\setminus\{w\}$ with $\deg v_1=1$;
    \item There exists $u_1\in \mathcal{N}(u)\triangle \mathcal{N}(v_1)=\mathcal{N}(u)\setminus\{v\}$ with $\deg u_1=1$.
\end{itemize}  Hence, $\Gamma$ is the graph in Figure \ref{fig:lemma:elementary-1}.

\begin{figure}[h]
    \centering
  \begin{tikzpicture}[scale=1.5]
\draw (-1,0)--(0,0)--(1,0)--(2,0)--(3,0);
\draw (1,0)--(1,1);
\node (u) at  (0,0) {$\bullet$};
\node (v) at  (1,0) {$\bullet$};
\node (w) at  (2,0) {$\bullet$};
\node (w1) at  (3,0) {$\bullet$};
\node (u1) at  (-1,0) {$\bullet$};
\node (v1) at  (1,1) {$\bullet$};
\node (u) at  (-0.1,0.2) {$u$};
\node (v) at  (0.8,0.2) {$v$};
\node (w) at  (2,0.2) {$w$};
\node (w1) at  (2.9,0.2) {$w_1$};
\node (u1) at  (-1,0.2) {$u_1$};
\node (v1) at  (1.2,0.9) {$v_1$};
\end{tikzpicture}
    \caption{A graph arising in the proof of Lemma \ref{lemma:elementary-1}}
    \label{fig:lemma:elementary-1}
\end{figure}
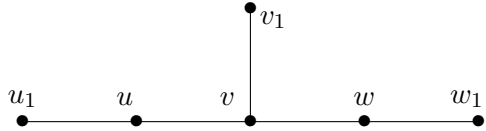

By letting $f(u_1):=f(w_1):=-1$,  $f(v):=1$ and $f(v_1):=f(u):=f(w):=0$, from \eqref{eq:important} we derive a contradiction.

\item $\deg u=2$, $\deg v=2$ and $\deg w=3$.

If there is only one vertex $x$ in  $\mathcal{N}(u)\triangle \mathcal{N}(w)$, then by \eqref{eq:repeat-key}, it must have degree $1$. Also, by construction, it is clear that $x\sim w$. Moreover, by \eqref{eq:repeat-key}, the only vertex $y\in \mathcal{N}(w)\setminus\{v,x\}$ has degree $1$. But this is a contradiction, since we must also have $y\sim u$.

Otherwise, there are three vertices in $\mathcal{N}(u)\triangle \mathcal{N}(w)$ and at least two of them have degree $2$. In this case, we reduce to Case 1.

\item $\deg u=2$, $\deg v=3$ and $\deg w=3$.

If $\deg x\ge2$ for all $x\in \mathcal{N}(u)\triangle \mathcal{N}(w)$, then by \eqref{eq:repeat-key},  there are at least two vertices $x$ and $y$ in $\mathcal{N}(u)\triangle \mathcal{N}(w)$ with degree $2$, and therefore we reduce to Case 1. 

There are two sub-cases left:
\begin{enumerate}
    \item[a)] If there exists $u_1\sim u$ with $\deg u_1=1$, then by \eqref{eq:repeat-key} there exists $v_1\sim v$ with $\deg v_1=1$ and there exists $w_1\sim w$ with $\deg w_1=1$. Therefore, $\Gamma$ has the same local structure as the graph in Figure \ref{fig:lemma:elementary-1b}.  By letting $f(u_1):=f(w_1):=-1$,  $f(v):=1$ and $f|_{V\setminus\{v,u_1,w_1\}}:=0$, from \eqref{eq:important} we get $1>\frac14(1+3+1)$, which is a contradiction.
    \begin{figure}[h]
    \centering
 \begin{tikzpicture}[scale=1.5]
\draw (-1,0)--(0,0)--(1,0)--(2,0)--(3,0);
\draw (1,0)--(1,1);
\node (u) at  (0,0) {$\bullet$};
\node (v) at  (1,0) {$\bullet$};
\node (w) at  (2,0) {$\bullet$};
\node (w1) at  (3,0) {$\bullet$};
\node (u1) at  (-1,0) {$\bullet$};
\node (v1) at  (1,1) {$\bullet$};
\node (u) at  (-0.1,0.2) {$u$};
\node (v) at  (0.8,0.2) {$v$};
\node (w) at  (1.8,0.2) {$w$};
\node (w1) at  (2.9,0.2) {$w_1$};
\node (u1) at  (-1,0.2) {$u_1$};
\node (v1) at  (1.2,0.9) {$v_1$};
\node (w2) at  (2.5,0.8) {$?$};
\draw[thick,dotted] (2.5,0.8)--(2,0);
\end{tikzpicture}
    \caption{A graph arising in the proof of Lemma \ref{lemma:elementary-1}}
    \label{fig:lemma:elementary-1b}
\end{figure}
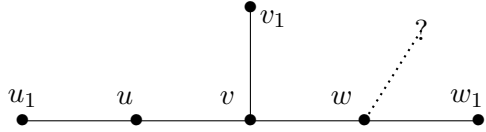
    \item[b)] If there exists $w_1\sim w$ with $\deg w_1=1$, then there exists $v_1\sim v$ with $\deg v_1=1$ and there exists $u_1\sim u$ with $\deg u_1=1$. Hence, as in the previous sub-case, $\Gamma$ has the same local structure as the graph in Figure \ref{fig:lemma:elementary-1b}, which brings to a contradiction.
\end{enumerate}

\item $\deg u=3$, $\deg v=2$ and $\deg w=3$.

If $\deg x\ge2$ for all $x\in \mathcal{N}(u)\triangle \mathcal{N}(w)$, then by \eqref{eq:repeat-key}  there are at least two vertices $x$ and $y$ in $\mathcal{N}(u)\triangle \mathcal{N}(w)$ of degree $2$, and thus we reduce to Case 1. 

Otherwise, there exists $u_1\sim u$ with $\deg u_1=1$. Applying Lemma \ref{lemma:key-tech} to the vertex pair $u_1$ and $v$, we derive a contradiction.

\item $\deg u=3$, $\deg v=3$ and $\deg w=3$.

Similar to the above cases, by repeatedly applying \eqref{eq:repeat-key} we can see that there exist $v_1\sim v$, $u_1\sim u$ and $w_1\sim w$ with $\deg v_1=\deg u_1=\deg w_1=1$. Therefore, $\Gamma$ has the same local structure as the graph in Figure \ref{fig:lemma:elementary-1c}. Similarly to Case 3, also in this case we derive a contradiction.

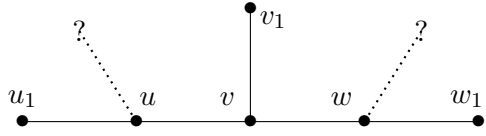
\begin{figure}[h]
    \centering
 \begin{tikzpicture}[scale=1.5]
\draw (-1,0)--(0,0)--(1,0)--(2,0)--(3,0);
\draw (1,0)--(1,1);
\node (u) at  (0,0) {$\bullet$};
\node (v) at  (1,0) {$\bullet$};
\node (w) at  (2,0) {$\bullet$};
\node (w1) at  (3,0) {$\bullet$};
\node (u1) at  (-1,0) {$\bullet$};
\node (v1) at  (1,1) {$\bullet$};
\node (u) at  (0.1,0.2) {$u$};
\node (v) at  (0.8,0.2) {$v$};
\node (w) at  (1.8,0.2) {$w$};
\node (w1) at  (2.9,0.2) {$w_1$};
\node (u1) at  (-1,0.2) {$u_1$};
\node (v1) at  (1.2,0.9) {$v_1$};
\node (w2) at  (2.5,0.8) {$?$};
\node (u2) at  (-0.5,0.8) {$?$};
\draw[thick,dotted] (-0.5,0.8)--(0,0);
\draw[thick,dotted] (2.5,0.8)--(2,0);
\end{tikzpicture}
    \caption{A graph arising in the proof of Lemma \ref{lemma:elementary-1}}
    \label{fig:lemma:elementary-1c}
\end{figure}
\end{enumerate}
\end{proof}

We now fix the following notations. Given a vertex $v$, we let
\begin{align*}
    \mathcal{N}_{2,3}(v)&:=\{w\sim v:2\le \deg w\le 3\},\\
     \mathcal{N}_{\geq 4}(v)&:=\{w\sim v:\deg w\ge 4\},\\
      \mathcal{N}_{2}(v)&:=\{w\sim v:\deg w= 2\}.
\end{align*}

\begin{lemma}\label{lemma:23}
Let $\Gamma$ be a connected graph on $N\geq 3$ vertices such that $\varepsilon>\frac{1}{2}$. Then, for each $v\in V$, $|\mathcal{N}_{2,3}(v)|\le 1$.
\end{lemma}

\begin{proof}Fix $v\in V$ and suppose the contrary, that is, $|\mathcal{N}_{2,3}(v)|\ge 2$.
\begin{enumerate}[{Case} 1:]
    \item $\mathcal{N}(v)=\mathcal{N}_{\geq 4}(v)\cup \mathcal{N}_{2,3}(v)$.   
    
In this case, for any two vertices $v_i,v_j\in \mathcal{N}_{2,3}(v)$, $$\min\limits_{x\in \mathcal{N}(v_i)\triangle \mathcal{N}(v_j)}\deg x=1,$$ because otherwise, we can reduce to  Lemma \ref{lemma:elementary-1}. Therefore, except for at most one vertex in $\mathcal{N}_{2,3}(v)$, any other vertex in  $\mathcal{N}_{2,3}(v)$ is adjacent to a vertex of degree $1$. Now, let $f(x):=-1$ if $\deg x=1$ and $x$ is adjacent to some vertex in $\mathcal{N}_{2,3}(v)$, let $f(v):=1$ and let $f(y):=0$ otherwise. Then, by \eqref{eq:important} we obtain 
\begin{align*}
\frac12+\frac14|\mathcal{N}_{\geq 4}(v)|&\ge\sum\limits_{w\in V}\frac{1}{\deg w}\left(\sum\limits_{v\in \mathcal{N}(w)}f(v)\right)^2\\&> \frac14\sum\limits_{w\in V} \deg w\cdot f(w)^2 \\&\ge \frac14 (\deg v+|\mathcal{N}_{2,3}(v)|-1)  \\&\ge \frac14 (|\mathcal{N}_{\geq 4}(v)|+3),
\end{align*}
which is a contradiction.
    \item There exists $ w\in \mathcal{N}(v)$ with $\deg w=1$. Then, by Lemma \ref{lemma:easycase}, $$\mathcal{N}(v)\setminus\{w\}=\mathcal{N}_{\geq 4}(v)\cup \mathcal{N}_{2,3}(v).$$
Similarly to the previous case, this implies that \begin{align*}
1+\frac14|\mathcal{N}_{\geq 4}(v)|&\ge\sum\limits_{w\in V}\frac{1}{\deg w}\left(\sum\limits_{v\in \mathcal{N}(w)}f(v)\right)^2\\&> \frac14\sum\limits_{w\in V} \deg w\cdot f(w)^2 \\&\ge \frac14 (\deg v+|\mathcal{N}_{2,3}(v)|)  \\&\ge \frac14 (|\mathcal{N}_{\geq 4}(v)|+4),
\end{align*}
which is a contradiction.
\end{enumerate}

\end{proof}

\begin{lemma}\label{lemma:1}
Let $\Gamma$ be a connected graph on $N\geq 3$ vertices such that $\varepsilon>\frac{1}{2}$. If $w\sim v$ and $\deg w=1$, then $\deg v\le 2$.
\end{lemma}
\begin{proof}

Assume the contrary, i.e., $w\sim v$, $\deg w=1$ and $\deg v\ge3$.  Then, similarly to the proof of Lemma \ref{lemma:23}, we have that for any $u\sim v$, there exists $u_1\sim u$ with $\deg u_1=1$. Therefore,
\begin{align*}
1&\ge\sum\limits_{w\in V}\frac{1}{\deg w}\left(\sum\limits_{v\in \mathcal{N}(w)}f(v)\right)^2\\&> \frac14\sum\limits_{w\in V} \deg w\cdot f(w)^2 \\&\ge \frac14 (\deg v+|\mathcal{N}(v)|-1)  \\&\ge \frac14 (2\deg v-1) >1,
\end{align*}
which is a contradiction.

\end{proof}

\begin{lemma}\label{lemma:step1}
Let $\Gamma$ be a connected graph on $N\geq 3$ vertices such that $\varepsilon>\frac{1}{2}$. Then, there exist three vertices $u\sim v\sim w$ with $\deg u\le 3$ and $\deg w\le 3$.
\end{lemma}
\begin{proof}
By Lemma \ref{lemma:deg3}, there exist $u_1\sim u_2$ with $\deg u_1\le 3$ and $\deg u_2\le 3$. If $v\sim u_2$, then by \eqref{eq:repeat-key} there exists  $w\in \mathcal{N}(u_1)\triangle \mathcal{N}(v)$ with $\deg w\le 3$.\newline
If $w\in \mathcal{N}(u_1)\setminus  \mathcal{N}(v)$, then we have $w\sim u_1\sim u_2$ with $\deg w\le 3$ and 
$\deg u_2\le 3$; while if $w\in \mathcal{N}(v)\setminus  \mathcal{N}(u_1)$, then we have $w\sim v\sim u_2$ with $\deg w\le 3$ and 
$\deg u_2\le 3$.
\end{proof}

\begin{lemma}\label{lemma:step2}
Let $\Gamma$ be a connected graph on $N\geq 3$ vertices such that $\varepsilon>\frac{1}{2}$. Then, there exist three vertices $ u\sim v\sim w$ such that $\deg u=1$, $\deg v=2$ and $\deg w\in\{2,3\}$.
\end{lemma}
\begin{proof}
Fix $u\sim v\sim w$ with $\deg u\le 3$ and $\deg w\le 3$ as in Lemma \ref{lemma:step1}. By Lemma \ref{lemma:easycase}, at most one between $u$ and $w$ has degree $1$, and by Lemma \ref{lemma:23}, at most one between $u$ and $w$ has degree $2$ or $3$. Therefore, we may assume that $\deg u=1$  and $\deg w\in\{2,3\}$. By Lemma \ref{lemma:1}, this implies that $\deg v=2$.
\end{proof}

We are now finally able to prove Theorem \ref{thm:part1}.
\begin{proof}[Proof of Theorem \ref{thm:part1}]  Assume by contradiction that there exists a connected graph $\Gamma$ on $N\geq 3$ vertices such that $\varepsilon>\frac{1}{2}$. Then, by Lemma \ref{lemma:step2}, there exist $u\sim v\sim w$ such that $\deg u=1$, $\deg v=2$ and $\deg w\in\{2,3\}$. By \eqref{eq:repeat-key}, there exists a vertex $w_1\in \mathcal{N}(w)\setminus\{v\}$ satisfying $\deg w_1=1$, and by Lemma \ref{lemma:1}, $\deg w=2$. Therefore $\Gamma$ is the path of length $3$ in Figure \ref{fig:3path}. But as we know this is a contradiction, since $\varepsilon=\frac{1}{2}$ for the path of length $3$, which is also the $1$-book graph.
\end{proof}

\subsection{Optimal cases}\label{section:optimal}

In the previous section we proved that $\varepsilon\leq\frac{1}{2}$ for any connected graph $\Gamma$ on $N\geq 3$ vertices. Hence, in order to prove Theorem \ref{thm:main}, it is left to prove that $\varepsilon=\frac{1}{2}$ if and only if $\Gamma$ is either a petal graph or a book graph. We dedicate this section to the proof of this claim, that we further split into three parts:
\begin{enumerate}
    \item In Section \ref{section:d3} we prove that, if the smallest vertex degree $d$ is greater than or equal to $3$, then $\varepsilon<\frac{1}{2}$. In fact, we prove an even stronger result, since we show that, in this case,
    \begin{equation*}
        \varepsilon\leq\frac{\sqrt{d-1}}{d}<\frac{1}{2}.
    \end{equation*}
    \item In Section \ref{section:d2} we show that, if $d=2$, then $\varepsilon=\frac{1}{2}$ if and only if $\Gamma$ is either a petal graph or a book graph.
    \item In Section \ref{section:d1} we show that, if $d=1$ and $\Gamma$ is not the $1$--book graph, then $\varepsilon<\frac{1}{2}$.
\end{enumerate}
Before, we observe that, by Lemma \ref{lemma:first} and Theorem \ref{thm:part1}, if a connected graph $\Gamma$ on $N\geq 3$ vertices is such that $\varepsilon=\frac{1}{2}$, then
\begin{equation}\label{eq:important-equa}
\sum\limits_{w\in V}\frac{1}{\deg w}\left(\sum\limits_{v\in \mathcal{N}(w)}f(v)\right)^2\geq \frac14\sum\limits_{w\in V} \deg w\cdot f(w)^2, \quad \forall f\in C(V)\setminus \{\mathbf{0}\}.
\end{equation}Moreover, with the same proof as Lemma \ref{lemma:key-tech}, one can prove the following
\begin{lemma}Let $\Gamma$ be a connected graph on $N\geq 3$ vertices such that $\varepsilon=\frac{1}{2}$. If $\mathcal{N}(u)\cap \mathcal{N}(v)\ne\emptyset$, then $\mathcal{N}(u)\triangle  \mathcal{N}(v)\ne\emptyset$ and 
\begin{equation}\label{eq:repeat-key-opti}
\sum_{w\in \mathcal{N}(u)\triangle \mathcal{N}(v)}\left(\frac{1}{\deg w}-\frac14\right)\geq \frac12.
\end{equation}
\end{lemma}

\subsubsection{The case $d\geq 3$}\label{section:d3}
In this section we prove the following
\begin{theo}\label{thm:d3}
For any connected graph $\Gamma$ on $N\geq 3$ vertices with smallest degree $d\geq 3$,
   \begin{equation*}
        \varepsilon\leq\frac{\sqrt{d-1}}{d}<\frac{1}{2}.
    \end{equation*}
\end{theo}

Before that theorem, we shall  prove several preliminary results.

\begin{lemma}
Let $\Gamma$ be a connected graph on $N\geq 3$ vertices with smallest degree $d\geq 3$ and $\varepsilon>\frac{\sqrt{d-1}}{d}$. Then,
\begin{equation}\label{eq:important-equa-d}
\sum\limits_{w\in V}\frac{1}{\deg w}\left(\sum\limits_{v\in \mathcal{N}(w)}f(v)\right)^2>\frac{d-1}{d^2} \sum\limits_{w\in V} \deg w\cdot f(w)^2, \quad \forall f\in C(V)\setminus \{\mathbf{0}\}.
\end{equation}
Moreover, if $\mathcal{N}(u)\cap \mathcal{N}(v)\ne\emptyset$, then $\mathcal{N}(u)\triangle  \mathcal{N}(v)\ne\emptyset$ and 
\begin{equation}\label{eq:repeat-key-opti-d}
\sum_{w\in \mathcal{N}(u)\triangle \mathcal{N}(v)}\left(\frac{1}{\deg w}-\frac{d-1}{d^2} \right)> 2\cdot \frac{d-1}{d^2} .
\end{equation}
\end{lemma}
\begin{proof}
The first claim follows directly from Lemma \ref{lemma:first}, while the second claim can be proved as Lemma \ref{lemma:key-tech}.
\end{proof}

\begin{lemma}\label{lemma16}
Let $\Gamma$ be a connected graph on $N\geq 3$ vertices with smallest degree $d\geq 3$ and $\varepsilon>\frac{\sqrt{d-1}}{d}$. If $\mathcal{N}(u)\cap \mathcal{N}(v)\ne\emptyset$, then $$|\{w\in \mathcal{N}(u)\triangle \mathcal{N}(v):\deg w\in\{d,d+1\}\}|\ge 2d-1.$$
\end{lemma}
\begin{proof}
Assume the contrary, then \begin{align*}
\sum_{w\in \mathcal{N}(u)\triangle \mathcal{N}(v)}\left(\frac{1}{\deg w}-\frac{d-1}{d^2} \right)&\le\sum_{\substack{w\in \mathcal{N}(u)\triangle \mathcal{N}(v)\\ \deg w\in\{d,d+1\}}}\left(\frac{1}{\deg w}-\frac{d-1}{d^2} \right)\\&< (2d-2) \left(\frac{1}{d}-\frac{d-1}{d^2}\right)=
2\cdot \frac{d-1}{d^2},
\end{align*}
which contradicts \eqref{eq:repeat-key-opti-d}.
\end{proof}
\begin{lemma}\label{lemma17}
Let $\Gamma$ be a connected graph on $N\geq 3$ vertices with smallest degree $d\geq 3$ and $\varepsilon>\frac{\sqrt{d-1}}{d}$. If $\mathcal{N}(u)\cap \mathcal{N}(v)\ne\emptyset$ and $$|\{w\in \mathcal{N}(u)\triangle \mathcal{N}(v):\deg w\in\{d,d+1\}\}|\le 2d,$$ then $|\{w\in \mathcal{N}(u)\triangle \mathcal{N}(v):\deg w=d\}|\ge 2d-3$.
\end{lemma}
\begin{proof}
If not, then \begin{align*}
&\sum_{w\in \mathcal{N}(u)\triangle \mathcal{N}(v)}\left(\frac{1}{\deg w}-\frac{d-1}{d^2} \right)\\
&\le\sum_{\substack{w\in \mathcal{N}(u)\triangle \mathcal{N}(v)\\ \deg w=d}}\left(\frac{1}{\deg w}-\frac{d-1}{d^2} \right)+\sum_{\substack{w\in \mathcal{N}(u)\triangle \mathcal{N}(v)\\ \deg w=d+1}}\left(\frac{1}{\deg w}-\frac{d-1}{d^2} \right)\\
&< (2d-4) \left(\frac{1}{d}-\frac{d-1}{d^2}\right)+ 2d \left(\frac{1}{d+1}-\frac{d-1}{d^2}\right)\\
&=\frac{2d-4}{d^2}+\frac{2}{d(d+1)}<\frac{2d-4}{d^2}+\frac{2}{d^2}
=2\cdot \frac{d-1}{d^2},
\end{align*}
which contradicts \eqref{eq:repeat-key-opti-d}.
\end{proof}
\begin{lemma}\label{lemma18}
Let $\Gamma$ be a connected graph on $N\geq 3$ vertices with smallest degree $d\geq 3$ and $\varepsilon>\frac{\sqrt{d-1}}{d}$. Then, there exist $u$ and $v$ with $\mathcal{N}(u)\cap \mathcal{N}(v)\ne\emptyset$ and $\deg u,\deg v\in\{d,d+1\}$.

\end{lemma}
\begin{proof}
Fix three vertices $u'\sim w'\sim v'$. By Lemma \ref{lemma16}, we either have $$|\{w\in \mathcal{N}(u')\setminus \mathcal{N}(v'):\deg w\in\{d,d+1\}\}|\ge d$$ or 
$$|\{w\in \mathcal{N}(v')\setminus \mathcal{N}(u'):\deg w\in\{d,d+1\}\}|\ge d.$$ 
Without loss of generality, we may assume that there are two vertices $u$ and $v$ in $\mathcal{N}(u')\setminus \mathcal{N}(v')$ with $\deg u,\deg v\in\{d,d+1\}$. This implies that $\mathcal{N}(u)\cap \mathcal{N}(v)\ne\emptyset$, hence $u$ and $v$ satisfy the claim.
\end{proof}
\begin{lemma}\label{lemma19}
Let $\Gamma$ be a connected graph on $N\geq 3$ vertices with smallest degree $d\geq 3$ and $\varepsilon>\frac{\sqrt{d-1}}{d}$. Then, there exist $u$ and $v$ with $\mathcal{N}(u)\cap \mathcal{N}(v)\ne\emptyset$ and $\deg u=\deg v=d$. 
\end{lemma}
\begin{proof}
Fix $u'$ and $v'$ that satisfy Lemma \ref{lemma18}. Then, $$|\mathcal{N}(u')\triangle \mathcal{N}(v')|\le \deg u'+\deg v'-2\le 2d.$$ By Lemma \ref{lemma17}, this implies that $$|\{w\in \mathcal{N}(u')\triangle \mathcal{N}(v'):\deg w=d\}|\ge 2d-3\ge3.$$ Hence, without loss of generality, we can assume that there are two vertices $u$ and $v$ in $\mathcal{N}(u')\setminus \mathcal{N}(v')$ with $\deg u=\deg v=d$. This proves the claim.
\end{proof}

We can now prove Theorem \ref{thm:d3}.

\begin{proof}[Proof of Theorem \ref{thm:d3}]
We first observe that the second inequality follows from the fact that the sequence
\begin{equation*}
    \left\{\frac{\sqrt{d-1}}{d}\right\}_{d\ge 2}=\left\{ \frac12,\frac{\sqrt{2}}{3},\frac{\sqrt{3}}{4},\frac25,\ldots \right\}
\end{equation*}
 is strictly decreasing. Hence, it is left to show that $\varepsilon\leq\frac{\sqrt{d-1}}{d}$. Assume, by contradiction, that $\varepsilon>\frac{\sqrt{d-1}}{d}$. By Lemma \ref{lemma19}, there exist  $u$ and $v$ such that $\mathcal{N}(u)\cap \mathcal{N}(v)\ne\emptyset$ and  $\deg u=\deg v=d$. This implies that $| \mathcal{N}(u)\triangle \mathcal{N}(v)|\le 2d-2$, but together with Lemma \ref{lemma16}, this brings to a contradiction.
\end{proof}

\subsubsection{The case $d=2$}\label{section:d2}
Here we prove the following
\begin{theo}\label{thm:d2}
For any connected graph $\Gamma$ on $N\geq 3$ vertices with smallest degree $d=2$,
   \begin{equation*}
        \varepsilon\leq\frac{1}{2},
    \end{equation*}with equality if and only if $\Gamma$ is either a petal graph or a book graph.
\end{theo}

Also in this section, we first prove several preliminary results. We start by showing that, if $\Gamma=C_N$ is the cycle graph on $N\geq 3$ vertices, then $\varepsilon=\frac{1}{2}$ if and only if either $N=3$ (in which case $\Gamma=C_3$ is the $1$--petal graph) or $N=6$ (in which case $\Gamma=C_6$ is the $2$--book graph).

\begin{prop}\label{prop:cycle}
If $\Gamma=C_N$ is the cycle graph on $N\geq 3$ vertices, then $\varepsilon=\frac{1}{2}$ if and only if $N\in \{3,6\}$.
\end{prop}
\begin{proof}As shown in \cite{Chung97}, for the cycle graph $C_N$ on $N\geq 3$ vertices, the eigenvalues are $1-\cos{(2\pi k/N)}$, for $k=0,\ldots,N-1$. We already know that $\varepsilon=\frac{1}{2}$ if $N\in \{3,6\}$ since, in this case, $\Gamma$ is either the $1$--petal graph or the $2$--book graph. Therefore, it is left to show that $\varepsilon<\frac{1}{2}$ for $N\geq 4$, $N\neq 6$. We consider three cases.
\begin{enumerate}
    \item If $N=4$, then $1$ is an eigenvalue, therefore $\varepsilon=0$.
    \item If $N=5$, by letting $k=1$ we can see that $1-\cos{(2\pi /5)}$ is an eigenvalue that has distance $\cos{(2\pi /5)}\approx 0.3$ from $1$. Hence, $\varepsilon<1/2$.
    \item Let $N> 6$ and let 
\begin{equation*}
    k\in \left\{\frac{N-2}{4},\frac{N-1}{4},\frac{N}{4}, \frac{N+1}{4}\right\}.
\end{equation*}Then,
\begin{equation*}
     \frac{2 k}{N}\leq \frac{2}{N}\cdot \frac{N+1}{4}=\frac{N+1}{2N}< \frac{2}{3},
\end{equation*}since $N> 3$, and 
\begin{equation*}
    \frac{2 k}{N}\geq \frac{2}{N}\cdot \frac{N-2}{4}=\frac{N-2}{2N}> \frac{1}{3},
\end{equation*}
since $N> 6$. Hence, 
\begin{equation*}
  \frac{2 k}{N}\in \left(\frac{1}{3},\frac{2}{3}\right),
\end{equation*}implying that 
\begin{equation*}
   \left|\cos{\frac{2\pi k}{N}}\right|<\frac{1}{2},
\end{equation*}
therefore $\varepsilon< 1/2$. 
\end{enumerate}
\end{proof}

\begin{prop}\label{prop:A1}
Let $\Gamma$ be a connected graph on $N\geq 3$ vertices with smallest degree $d=2$ and  $\varepsilon=\frac{1}{2}$. If three vertices are such that $u\sim v\sim w$ and $\deg u=\deg v=\deg w=2$, then $\Gamma=C_3$ or $C_6$.
\end{prop}
\begin{proof}
If $u\sim w$, then clearly $\Gamma=C_3$. If  $u\not\sim w$, then \eqref{eq:repeat-key-opti} applied to $ \mathcal{N}(u)\triangle\mathcal{N}(w)$ implies that the vertices $u_1\in  \mathcal{N}(u)\setminus\{v\}$ and $w_1\in \mathcal{N}(w)\setminus\{v\}$ are such that $\deg u_1=\deg w_1=2$. If $u_1=w_1$, then clearly $\Gamma=C_4$. Otherwise, we can repeat the process and obtain any cycle graph $C_N$. By Proposition \ref{prop:cycle}, the claim follows.
\end{proof}

\begin{lemma}\label{lemmaA3}Let $\Gamma$ be a connected graph on $N\geq 3$ vertices with smallest degree $d= 2$ and $\varepsilon=\frac{1}{2}$. 
Then, there exist two adjacent vertices $u\sim v$ such that $\deg u=\deg v=2$.
\end{lemma}
\begin{proof}We illustrate this proof in Figure \ref{fig:A3}.\newline
Fix $u\sim v\sim w$ such that $\deg v=2$. Then, by \eqref{eq:repeat-key-opti}, there exists $w_1\in \mathcal{N}(u)\triangle \mathcal{N}(w)$ with $\deg w_1\in\{2,3\}$. Without loss of generality, we assume that $w_1\sim w$. Again by \eqref{eq:repeat-key-opti}, there are at least two 
vertices 
in $\mathcal{N}(v)\triangle \mathcal{N}(w_1)$ of degree $2$. If $\deg u=2$, then we have done. If $\deg u\ge 3$, then $\deg w_1=3$, $\deg u\le 4$ and $\deg x=2$ for any $x\in \mathcal{N}(w_1)\setminus\{w\}$. Now, let $x_1,x_2\in\mathcal{N}(w_1)\setminus\{w\}$. By applying  \eqref{eq:repeat-key-opti}  to $x_1$ and $x_2$, we infer that there exists a  vertex adjacent to $x_1$ of degree 2. Since also $\deg x_1=2$, this proves the claim.
\end{proof}

\begin{figure}[h]
    \centering
    \begin{tikzpicture}[scale=1.5]
\draw (0,0)--(1,0)--(2,0)--(3,0);
\draw[dashed] (0.2,0.9)--(0,0)--(-1,0);\draw[dashed] (1.8,0.8)--(2,0);
\draw (3.2,0.9)--(3,0)--(4,0);
\node (u) at  (0,0) {$\bullet$};
\node (v) at  (1,0) {$\bullet$};
\node (w) at  (2,0) {$\bullet$};
\node (w1) at  (3,0) {$\bullet$};
\node (x1) at  (4,0) {$\bullet$};
\node (x2) at  (3.2,0.9) {$\bullet$};
\node (u) at  (-0.1,0.2) {$u$};
\node (v) at  (0.8,0.2) {$v$};
\node (w) at  (2,0.2) {$w$};
\node (w1) at  (2.9,0.2) {$w_1$};\node (x1) at  (3.9,0.2) {$x_1$};\node (x2) at  (3,0.8) {$x_2$};
\end{tikzpicture}
    \caption{The vertices in the proof of Lemma \ref{lemmaA3}}
    \label{fig:A3}
\end{figure}
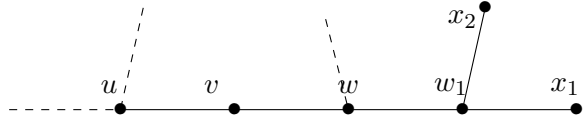

\begin{lemma}\label{lemmaA4}
Let $\Gamma$ be a connected graph on $N\geq 3$ vertices with smallest degree $d= 2$ and $\varepsilon=\frac{1}{2}$. For any vertex $v$ of degree 2, there exists $u\sim v$ with $\deg u=2$.
\end{lemma}
\begin{proof}
Suppose the contrary, then there are $u\sim v\sim w$ with $\deg u,\deg w\ge3$ and $\deg v=2$.

\begin{enumerate}[{Claim }1:]
\item For all $x\in \mathcal{N}(u)\triangle \mathcal{N}(w)$, $\deg x\ge 3$.

If not, without loss of generality we can assume that there exists $x\in \mathcal{N}(u)\setminus \mathcal{N}(w)$ with $\deg x=2$. This implies that  $$\mathcal{N}(v)\triangle \mathcal{N}(x)=\{w,y\},$$ for some vertex $y$. By the assumption,  $\deg y\ge2$ and $\deg w\ge 3$. Hence, applying  \eqref{eq:repeat-key-opti} to $\mathcal{N}(v)\triangle \mathcal{N}(x)$ implies  $$\left(\frac{1}{2}-\frac14\right)+\left(\frac{1}{3}-\frac14\right)\ge\left(\frac{1}{\deg y}-\frac14\right)+\left(\frac{1}{\deg w}-\frac14\right)\ge\frac12,$$ which is a contradiction.
\item There are at least six vertices in $ \mathcal{N}(u)\triangle \mathcal{N}(w)$ of degree 3, and thus $\deg u+\deg w\ge 8$. 

If there are at most five vertices in $ \mathcal{N}(u)\triangle \mathcal{N}(w)$ of degree 3, then the other vertices in $ \mathcal{N}(u)\triangle \mathcal{N}(w)$ have  degree at least 4. Applying  \eqref{eq:repeat-key-opti} to $ \mathcal{N}(u)\triangle\mathcal{N}(w)$, we then obtain
$$5 \left(\frac13-\frac14\right)\ge \sum_{x\in \mathcal{N}(u)\triangle \mathcal{N}(w)}\left(\frac{1}{\deg x}-\frac14\right)\geq \frac12,$$ which is a contradiction. 
\item Assume that there exists $w_1\sim w$ with $\deg w_1=3$. Then, $x_1,x_2\in \mathcal{N}(w_1)\setminus\{w\}$ are such that $\deg x_1=\deg x_2=2$.

If not, then without loss of generality we can assume that $\deg x_1\ge 3$. Applying  \eqref{eq:repeat-key-opti} to $ \mathcal{N}(v)\triangle\mathcal{N}(w_1)$, this implies
\begin{align*}
&\frac{5}{12}=\left(\frac13-\frac14\right)+\left(\frac13-\frac14\right)+\left(\frac12-\frac14\right)\\
\ge& \left(\frac{1}{\deg u}-\frac14\right)+\left(\frac{1}{\deg x_1}-\frac14\right)+\left(\frac{1}{\deg x_2}-\frac14\right)\\ =&\sum_{x\in \mathcal{N}(v)\triangle \mathcal{N}(w_1)}\left(\frac{1}{\deg x}-\frac14\right)\geq \frac12,
\end{align*}
which is a contradiction.
\end{enumerate}
  Now, if $x_1\sim x_2$, then by letting $f(x_1):=-1$, $f(w_1):=1$ and $f:=0$ otherwise, \eqref{eq:repeat-key-opti} brings to a contradiction. Hence, $x_1\not\sim x_2$. By applying \eqref{eq:repeat-key-opti} again, we can infer that there exist  $y_1\sim x_1$ and $y_2\sim x_2$ such that $\deg y_1=\deg y_2=2$. 
If $y_1=y_2$, then we reduce to Proposition \ref{prop:A1}. If $y_1\ne y_2$, then a similar reasoning implies $y_1\not\sim y_2$. By applying then \eqref{eq:repeat-key-opti} to $ \mathcal{N}(y_1)\triangle\mathcal{N}(w_1)$, we infer that there exists $z_1\sim y_1$ with $\deg z_1=2$, and therefore we again reduce to Proposition \ref{prop:A1}.  In any case, we obtain a contradiction.
\end{proof}

\begin{prop}\label{prop:A2}Let $\Gamma$ be a connected graph on $N\geq 3$ vertices with smallest degree $d= 2$ and $\varepsilon=\frac{1}{2}$. 
If there exist $u\sim v\sim w$ such that $\deg u=\deg v=2$ and $\deg w=3$, then $\Gamma$ is the book graph on $N=8$ vertices.
\end{prop}
\begin{proof}
If $w\sim u$, by letting $f(w):=1$, $f(u):=-1$, and $f:=0$ otherwise, \eqref{eq:important-equa} implies that $$\frac12+\frac13+\frac{1}{\deg w_1}\ge\frac14(2+3),$$ therefore the vertex  $w_1\in \mathcal{N}(w)\setminus\{u,v\}$ has degree $2$. Similarly, there exists $x\sim w_1$ with $\deg x=2$. By letting now $f(w):=2$, $f(u):=f(v):=-1$, $f(w_1):=1$, $f(x):=-1$ and $f:=0$ otherwise, then \eqref{eq:important-equa} implies that $$2\cdot \frac12+\frac13+\frac12+\frac12+\frac13\ge \frac14(2+2+3\cdot 2^2+2+2),$$ which is a contradiction. Therefore, $w\not\sim u$.\newline 
Now, if the vertex $u_1\in \mathcal{N}(u)\setminus \{v\}$ has degree $2$, then we reduce to Proposition \ref{prop:A1}. Thus, $\deg u_1\ge 3$. Applying \eqref{eq:repeat-key-opti}  to the vertices $u$ and $w$, we infer that the two vertices $w_1,w_2\in \mathcal{N}(w)\setminus\{v\}$ satisfy $\deg w_1=\deg w_2=2$. 
Again, applying \eqref{eq:repeat-key-opti} to the vertex pairs $(v,w_1)$ and  $(v,w_2)$, we infer that the vertices $x_1\sim w_1$ and $x_2\sim w_2$ satisfy $\deg x_1=\deg x_2=2$.\newline
If $\mathcal{N}(x_1)\cap \mathcal{N}(x_2)=\emptyset$, then the vertices $y_1\sim x_1$ and $y_2\sim x_2$ have to satisfy $\deg y_1,\deg y_2\ge 3$. 
By letting $f(w):=1$, $f(x_1):=f(x_2):=-1$ and $f:=0$ otherwise, \eqref{eq:important-equa} implies 
$$\frac13+\frac13+\frac12\ge\frac14(3+2+2),$$ which is a contradiction. Hence, $\mathcal{N}(x_1)\cap \mathcal{N}(x_2)\neq \emptyset$ and, similarly, also $\mathcal{N}(u)\cap \mathcal{N}(x_1)\ne\emptyset$ and  $\mathcal{N}(u)\cap \mathcal{N}(x_2)\ne \emptyset$.\newline
In particular, there exists $z$ such that $z\sim u$, $z\sim x_1$ and $z\sim x_2$. By letting now $f(w):=1$, $f(u):=f(x_1):=f(x_2):=-1$ and $f:=0$ otherwise, by  \eqref{eq:important-equa} we obtain
$$\frac{1}{\deg z}\cdot 3^2\ge\frac14(3+2+2+2),$$ which implies that $\deg z\le 4$.\newline
We now claim that $\deg z=3$.  Suppose the contrary, then $\deg z= 4$. Let $z_1\in \mathcal{N}(z)\setminus\{u,x_1,x_2\}$. If there exists $z_2\sim z_1$ of degree $2$, then $f(z_2):=f(u):=-1$, $f(z):=1$ and $f:=0$ otherwise together with  \eqref{eq:important-equa} gives $$\frac13+3\cdot\frac12\ge\frac14(2+4+2),$$  which is a contradiction. Therefore, for any $a\in \mathcal{N}(z_1)$, $\deg a\ge 3$. By \eqref{eq:repeat-key-opti} applied to $z_2$ and $v$, we infer that there exist at least three vertices of degree 3 in $\mathcal{N}(z_1)$, and $\deg z_1\ge 4$. Let $b_1,b_2\in \mathcal{N}(z_1)$ be such that $\deg b_1=\deg b_2=3$. Then, \eqref{eq:repeat-key-opti} applied to $b_1$ and $b_2$ implies that there is at least one vertex of degree $2$ in $\mathcal{N}(b_1)\triangle \mathcal{N}(b_2)$. Without loss of generality, we can assume that there exists  $b'\in \mathcal{N}(b_1)\setminus  \mathcal{N}(b_2)$ such that $\deg b'=2$. By Lemma \ref{lemmaA4}, there exists $c'\sim b'$ with $\deg c'=2$. But this is a contradiction, since we proved that all vertices adjacent to $b_1$ must  have degree 2.\newline
This proves that $\deg z=3$, which implies that $\Gamma$ must be the book graph on $N=8$ vertices.
\end{proof}

\begin{lemma}\label{lemmaA5}Let $\Gamma$ be a connected graph on $N\geq 3$ vertices with smallest degree $d= 2$ and $\varepsilon=\frac{1}{2}$. 
If $\Gamma$ is not in the classes described in Proposition \ref{prop:A1} and Proposition \ref{prop:A2}, then there is no vertex of degree $3$.
\end{lemma}
\begin{proof}
Assume the contrary, that is, there exists a vertex $u$ with $\deg u=3$. We first claim that, for any $v\in \mathcal{N}(u)$, $\deg v\ge 3$. This is true because, otherwise, there would be a vertex $v\sim u$ with $\deg v=2$, and by Lemma \ref{lemmaA4},
 also a vertex $w\sim v$ with $\deg w=2$. Hence, $\Gamma$ would be in the class described in Proposition \ref{prop:A2}, but we are assuming that this is not the case. This shows that $\deg v\ge 3$ for any $v\in \mathcal{N}(u)$.\newline
Now, we prove that 
there are no vertices $v$ and $w$ such that $u\sim v\sim w$  and $\deg w\in\{2,3\}$. This follows from the fact that, otherwise, applying \eqref{eq:repeat-key-opti} to $u$ and $w$ together with the above remarks implies that there exists $w_1\sim w$ with $\deg w_1=2$, while $\deg w=3$, which is a contradiction. \newline
But on the other hand, for any $x,y\in \mathcal{N}(u)$,  by \eqref{eq:repeat-key-opti} there exists $w\in \mathcal{N}(x)\triangle \mathcal{N}(y)$ with $\deg w\in\{2,3\}$, which contradicts the above argument.
\end{proof}

We are now able to prove Theorem \ref{thm:d2}.

\begin{proof}[Proof of Theorem \ref{thm:d2}]
We shall illustrate this proof in Figure \ref{fig:thmd2} below.\newline
Observe that, if $\Gamma$ is in one of the classes described in Proposition \ref{prop:A1} and Proposition \ref{prop:A2}, then the claim follows. Therefore, we can assume that this is not the case. We can then apply Lemma \ref{lemmaA5} and infer that there is no vertex of degree $3$. Also, if all vertices have degree $<3$, from the fact that $d=2$ it follows that all vertices have degree $2$, but this is not possible, since we are assuming that $\Gamma$ does not satisfy the condition in  Proposition \ref{prop:A1}. Therefore, we can assume that there exists a vertex with degree $\ge 4$.\newline Moreover, we can also assume that there exists one vertex $u$ with $\deg u\geq 4$ such that $\mathcal{N}_2(u)\neq \emptyset$. This can be seen by the fact that, if not, then all vertices of degree $2$ would only be adjacent to vertices of degree $2$, contradicting the fact that $\Gamma$ is connected.\newline Now, from the fact that $\mathcal{N}_2(u)\neq \emptyset$, we can infer two facts about $\mathcal{N}(u)=\mathcal{N}_2(u)\sqcup \mathcal{N}_{\geq 4}(u):$
\begin{enumerate}
    \item For any $v\in \mathcal{N}_2(u)$, there exists a unique $w\sim v$ with $\deg w=2$. This follows immediately from Lemma \ref{lemmaA4} and from the fact that $\deg u\geq 4$.
    \item If $\mathcal{N}_{\geq 4}(u)\ne\emptyset$, then for any $v\in \mathcal{N}_{\geq 4}(u)$ there exists $w\sim v$ with $\deg w=2$. This can be shown by fixing $v'\in \mathcal{N}_{2}(u)$ and then applying \eqref{eq:repeat-key-opti} to $ \mathcal{N}(v')\triangle \mathcal{N}(v)$.
\end{enumerate}

As a consequence of the first fact, we can write
$$\mathcal{N}_{2}(u)=\{v_1,\ldots,v_k,v_{k+1},\ldots,v_{k+2l}\},$$
where:
\begin{itemize}
    \item For each $i=1,\ldots,k$, there exists a unique $w_i\in V\setminus\mathcal{N}_2(u)$ such that $v_i\sim w_i$ and $\deg w_i=2$;
    \item $v_{k+1}\sim v_{k+2},\ldots,v_{k+2l-1}\sim v_{k+2l}$.
\end{itemize}

Now, for $i=1,\ldots,k$, since $v_i\sim w_i$ and $\deg v_i=\deg w_i=2$, we can infer that $x_i\in \mathcal{N}(w_i)\setminus \{v_i\}$ must have degree $\geq 4$, because otherwise, $\Gamma$ would satisfy the assumption of Proposition \ref{prop:A1}, but we are assuming that this is not the case. Therefore, in particular, $w_1,\ldots,w_k$ must be pairwise distinct, while $x_1,\ldots,x_k$ do not need to be pairwise distinct. \newline
Up to relabeling $x_1,\ldots x_k$, we assume that
\begin{align*}
    x'_1&:=x_1=\ldots=x_{k_1},\\
    x'_2&:=x_{k_1+1}=\ldots=x_{k_1+k_2},\\
    &\ldots \\
    x'_{p+q}&:=x_{k_1+\ldots+k_{p+q-1}+1}=\ldots=x_{k_1+\ldots+k_{p+q}}=x_k,
\end{align*}where:
\begin{itemize}
    \item $k_1,\ldots,k_{p+q}$ are such that $k=k_1+\ldots+k_{p+q}$,
    \item $x_1',\ldots,x'_{p+q}$ are pairwise distinct,
    \item $x_1',\ldots,x'_{p}\not\in \mathcal{N}(u)$ while $x'_{p+1},\ldots,x'_{p+q}\in \mathcal{N}(u)$,
    \item $q$ is such that $q\le \deg u-k-2l$.
\end{itemize}

Following the above notation, we also relabel the vertices $v_1,\ldots,v_k$ and $w_1,\ldots,w_k$ as
\begin{align*}
   &v_1,\ldots,v_{k_1},
   \ldots, 
   v_{k_1+\ldots+k_{p+q}},
   \\ &w_1,\ldots,w_{k_1},
   \ldots, 
 w_{k_1+\ldots+k_{p+q}}.
\end{align*}
Now, we write

$$\mathcal{N}_{\geq 4}(u)\setminus \{x'_{p+1},\ldots,x'_{p+q}\}=\{y_1,\ldots,y_r\}.$$ 

Then, as observed above, for each $i=1,\ldots, r$ there exists $y_i'\sim y_i$ with $\deg y_i'=2$. Furthermore, by Lemma \ref{lemmaA4}, for each $i=1,\ldots,k$ there also exists $z_i\sim y_i'\sim y_i$ with $\deg z_i=\deg y_i'=2$. Clearly, $y_1',\ldots,y_r'$ are pairwise distinct, while $z_1,\ldots,z_r$ may have some overlap with $y_1',\ldots,y_r'$.\newline 
We illustrate the above vertices in Figure \ref{fig:thmd2} and we observe that
\begin{align*}
    \mathcal{N}(u)&=\mathcal{N}_2(u)\sqcup \mathcal{N}_{\geq 4}(u)\\ &=\{v_1,\ldots,v_k,v_{k+1},\ldots,v_{k+2l}\}\sqcup \{x'_{p+1},\ldots,x'_{p+q}\} \sqcup \{y_1,\ldots,y_r\},
\end{align*}
where the above vertices are all pairwise distinct, therefore $$\deg u=k+2l+q+r.$$

\begin{figure}
    \centering
    \begin{tikzpicture}[scale=1.8]
\draw (0,0)--(1,2)--(2,2)--(3,1.5);
\draw (0,0)--(1,1.5)--(2,1.5)--(3,1.5);
\draw (0,0)--(1,0.8)--(2,0.8)--(3,1.5);
\draw (0,0)--(-1,0.5)--(-1,0.1)--(0,0)--(-1,-0.4)--(-1,-0.8)--(0,0);
\draw (0,0)--(1,0)--(3,0)--(4.5,0)--(3,-0.4)--(1,-0.4)--(0,0);
\node(1) at (1,1.2) {$\vdots$};
 \node (u) at  (0,0) {$\bullet$};
\node (v) at  (1,2) {$\bullet$};
\node (v) at  (1,1.5) {$\bullet$};
\node (v) at  (1,0.8) {$\bullet$};
\node (w) at  (2,2) {$\bullet$};
\node (w) at  (2,1.5) {$\bullet$};
\node (w) at  (2,0.8) {$\bullet$};
\node (v) at  (-1,0.5) {$\bullet$};
\node (v) at  (-1,0.1) {$\bullet$};
\node (v) at  (-1,-0.4) {$\bullet$};
\node (v) at  (-1,-0.8) {$\bullet$};
\node (v) at  (1-0.2,2-0.1) {$v_1$};
\node (v) at  (1,1.5+0.14) {$v_2$};
\node (v) at  (2-0.2,2-0.1) {$w_1$};
\node (v) at  (2,1.5+0.14) {$w_2$};
\node (v) at  (1-0.15,0.8+0.14) {$v_{k_1}$};
\node (v) at  (-1.3,0.5) {$v_{k+1}$};
\node (v) at  (-1.3,0.1) {$v_{k+2}$};
\node (v) at  (-1.5,-0.4) {$v_{k+2l-1}$};
\node (v) at  (-1.4,-0.8) {$v_{k+2l}$};
\node(1) at (-1,-0.1) {$\vdots$};
\node (x) at  (3,1.5) {$\bullet$};
\node (x) at  (3.2,1.5) {$x'_1$};
\node(1) at (1,0.5) {$\vdots$};
\node(v) at (1,0) {$\bullet$};
\node(v) at (1.1,0.2) {$v_{k_1+\ldots+k_{p-1}+1}$};
\node(v) at (2.9,0.2) {$w_{k_1+\ldots+k_{p-1}+1}$};
\node(1) at (1,-0.2) {$\vdots$};
\node(v) at (1,-0.4) {$\bullet$};
\node(v) at (1.1,-0.6) {$v_{k_1+\ldots+k_{p}}$};
\node(w) at (2.9,-0.6) {$w_{k_1+\ldots+k_{p}}$};\node(w) at (3,-0.4) {$\bullet$};
\node(x) at (4.3,0.2) {$x'_p$};
\draw (4.6,0.5)--(4.5,0)--(4.8,0);
\node (w) at  (3,0) {$\bullet$};
\node (x) at  (4.5,0) {$\bullet$};
\node (u) at  (0.05,0.3) {$u$};
\node(y) at (0,2) {$\bullet$};
\draw (0.3,1.9)--(0,2)--( 0.3,1.6);
\node(y') at (-1,2) {$\bullet$};
\node(y) at (-0.1,1.9) {$y_1$};
\node(y') at (-1.15,1.8) {$y_1'$};
\node(1) at (-1.3,1.3) {$\vdots$};
\node(y) at (-0.5,1.6) {$\bullet$};
\draw (-0.6,1.8)--(-0.5,1.6)--(-0.2,1.5);
\node(y') at (-1.5,1.6) {$\bullet$};
\node(y) at (-0.5-0.1,1.6-0.1) {$y_2$};
\node(y') at (-1.5-0.15,1.6-0.1) {$y_2'$};
\node(y) at (-1,1) {$\bullet$};
\node(y') at (-2,1) {$\bullet$};
\node(y) at (-0.8,1.1) {$y_r$};
\draw (-2,1)--(-2.3,0.5);
\node(y') at (-2.2,1.2) {$y_r'$};
\draw (-2,1.8)--(-1.5,1.6)--(-0.5,1.6)--(0,0)--(0,2)--(-1,2)--(-1.5,2);
\draw (0,0)--(-1,1)--(-2,1);
\draw (0,0) to[out=-60,in=150] (4.5,-1.1);
\draw (0,0)--(0.8,-1)--(3,-1.3)--(4.5,-1.1);
\node (v) at (0.8,-1) {$\bullet$};
\node (v) at (0.5,-1.3) {$\bullet$};
\node (w) at (3,-1.3) {$\bullet$};
\node (w) at (3,-1.8) {$\bullet$};\draw (0,0)--(0.5,-1.3)--(3,-1.8)--(4.5,-1.1);
\node (1) at (2.2,-1.4) {$\cdots$};
\node(x') at (4.5,-1.1) {$\bullet$};
\draw (4.5,-1.1)--(4.8,-1.5);
\draw (0,0)--(-1.95,-2)--(-0.5,-2)--(0,-1.6)--(0,0);\draw (0,0)--(-0.8,-1.1)--(-1.6,-1.8)--(-1.95,-2);
\node (w) at (-1.6,-1.8) {$\bullet$};
\node (w) at (-0.5,-2) {$\bullet$};
\node (v) at (0,-1.6) {$\bullet$};
\node (v) at (-0.8,-1.1) {$\bullet$};
\node (x) at (-1.95,-2) {$\bullet$};
\node(x') at (4.5,-0.8) {$x_{p+1}'$};
\node(v) at (1.5,-0.93) {$v_{k_1+\ldots+k_{p}+1}$};
\node(w) at (3,-1.1) {$w_{k_1+\ldots+k_{p}+1}$};
\node (v) at (1.1,-1.2) {$v_{k_1+\ldots+k_{p+1}}$};
\node (w) at (3,-2) {$w_{k_1+\ldots+k_{p+1}}$};
\node (1) at (1.2,-1.9) {$\cdots\cdots$};
\node (1) at (-0.5,-1.6) {$\cdots$};
\node (v) at (0.9,-1.6) {$v_{k_1+\ldots+k_{p+q-1}+1}$};
\node (w) at (0.4,-2.15) {$w_{k_1+\ldots+k_{p+q-1}+1}$};
\node (v) at (-0.8-0.75,-1.1) {$v_{k_1+\ldots+k_{p+q}}$};
\node (w) at (-1.6+0.75,-1.8) {$w_{k_1+\ldots+k_{p+q}}$};
\node (x') at (-1.95-0.2,-2.1) {$x_{p+q}'$};
\draw (-1.95,-2)--(-2,-1.5);
\end{tikzpicture}
    \caption{Illustration of the proof of Theorem \ref{thm:d2}}
    \label{fig:thmd2}
\end{figure}
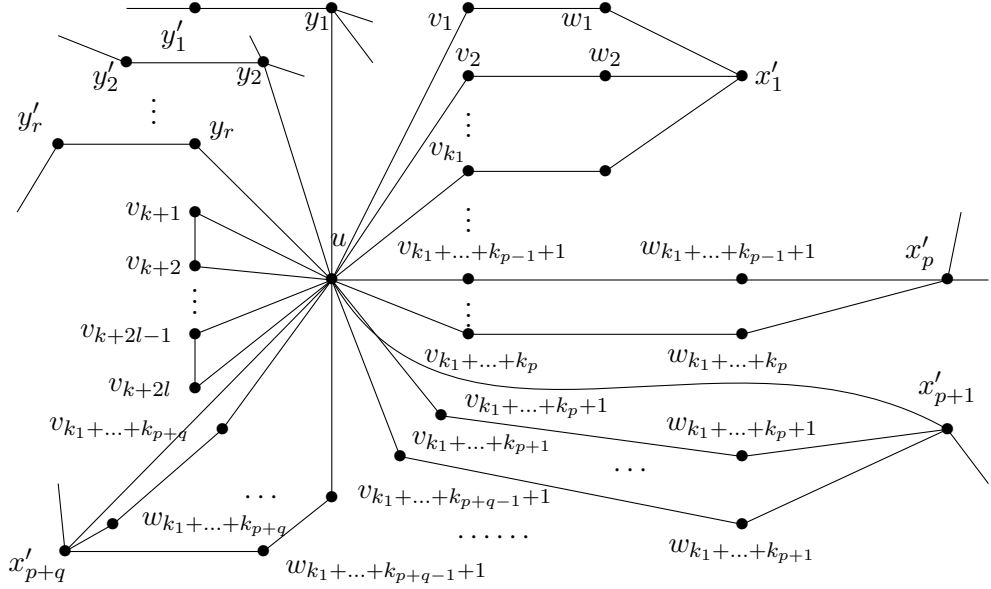

Now, by letting 
\begin{align*}
    &f(u):=2,\\
    &f(v_{k+1}):=\ldots:=f(v_{k+2l}):=f(w_1):=\ldots:=f(w_k):=-1,\\
    & f(v_1):=\ldots:=f(v_k):=1 \quad \text{if }k\le 2l,\\
    & f(v_1):=\ldots:=f(v_{2l}):=1,\\
    &f(v_{2l+t}):=(-1)^t \quad \text{for }t=1,\ldots,k-2l, \quad \text{if }2l<k,\\
    &f(y_1'):=\ldots:=f(y_r'):=-1,\\
    &f:=0 \quad \text{otherwise}
\end{align*}and
$$ c_{k,l}:=\begin{cases}(2l-k)^2,&\text{ if }2l\ge k,\\
1,&\text{ if }2l< k\text{ and }k\text{ is odd},\\
0,&\text{ if }2l< k\text{ and }k\text{ is even},
\end{cases}$$
from \eqref{eq:important-equa} we can infer that
\begin{align*}
&\frac{c_{k,l}}{\deg u}+\frac{2k+2l}{2}+\sum_{i=1}^p\frac{k_i^2}{\deg x_i'}+\sum_{i=p+1}^{p+q}\frac{(k_i-2)^2}{\deg x_i'}+r\cdot\frac12+\sum_{i=1}^r\frac{1}{\deg y_i}\\ =~&\sum\limits_{w\in V}\frac{1}{\deg w}\left(\sum\limits_{v\in \mathcal{N}(w)}f(v)\right)^2
\\\ge~&
\frac14\sum\limits_{w\in V} \deg w\cdot f(w)^2 \\=~& \frac14 (\deg u\cdot 2^2
+2(2k+2l)+2r
)  \\=~& 
\deg u+k+l+\frac r2=2k+3l+q+\frac32r.
\end{align*}
We now use the following known facts:
\begin{itemize}
    \item $\deg y_i\ge 4$, for each $i=1,\ldots,r$;
    \item $\deg x_i'\ge\max\{k_i,4\}$, for $i=1,\ldots,p+q$;
    \item $(k_i-2)^2\le k_i^2$ (since $k_i\ge 1$), for $i=p+1,\ldots,p+q$.
\end{itemize}
Putting everything together, we obtain that
\begin{align*}
&2k+3l+q+\frac32r
\\ \le~&\frac{c_{k,l}}{\deg u}+k+l+\sum_{i=1}^{p+q}\frac{k_i^2}{\deg x_i'}+r\cdot \frac12+\sum_{i=1}^r\frac{1}{4}
\\ \le~&\frac{c_{k,l}}{k+2l+q+r}+k+l+\sum_{i=1}^{p+q}\frac{k_i^2}{k_i}+\frac34r,
\end{align*}
which implies 
$$k+2l+q+\frac34 r\le \frac{c_{k,l}}{k+2l+q+r}+\sum_{i=1}^{p+q}k_i=k+\frac{c_{k,l}}{k+2l+q+r}.$$
Therefore, $$2l+q+\frac34 r\le \frac{c_{k,l}}{k+2l+q+r}\le\begin{cases}\frac{(2l-k)^2}{k+2l}\le 2l&\text{ if }2l\ge k,\\
\frac{1}{\deg u}\le \frac14&\text{ if }2l< k,
\end{cases}$$
which implies that $r=q=0$, either $l=0$ or $k=0$ (but they cannot be both zero), and $\deg x_i'=k_i\ge 4$, for $i=1,\ldots,p$.\newline 
We then have two cases:
\begin{enumerate}
    \item If $k=0$, then $\Gamma$ is a petal graph with $2l+1$ vertices.
\item If $l=0$, then we can write
$$\Gamma=(\Gamma_1\sqcup\ldots\sqcup\Gamma_p)/\{u_1,\ldots,u_p\},$$
where for $i=1,\ldots,p$, $\Gamma_i$ as a $k_i$-book graph and $u_i$ as one of its vertices of degree $k_i$, and the above notation means that $\Gamma$ is given by the union of $\Gamma_1,\ldots,\Gamma_p$ by identifying $u_1,\ldots,u_p$ as one vertex, which is $u$.\newline
In this case, by letting $f(x_1'):=2$, $f(w_i):=(-1)^i$ for $i=1,\ldots,k_1$, $f(v_1):=\ldots:=f(v_k):=-1$ and $f:=0$ otherwise, by \eqref{eq:important-equa} we can infer that $\deg u=k_1$ (with the same proof that we used for inferring that $\deg x_1'=k_1$). As a consequence, $p=1$, therefore $\Gamma$ is a book graph on $N=2k+2$ vertices.
\end{enumerate}
This proves the claim.
\end{proof}

\subsubsection{The case $d=1$}\label{section:d1}

Given $N\geq 3$, we let $P_N$ denote the path graph on $N$ vertices, and we observe that $P_4$ coincides with the $1$--book graph. It is left to prove the following
\begin{theo}
For any connected graph $\Gamma$ on $N\geq 3$ vertices with smallest degree $d=1$, if $\Gamma\neq P_4$ then
   \begin{equation*}
        \varepsilon<\frac{1}{2}.
    \end{equation*}
\end{theo}
\begin{proof}
Suppose the contrary, then there exists a connected graph $\Gamma\neq P_4$ on $N\geq 3$ vertices with $\varepsilon=\frac{1}{2}$ and $d=1$. We fix such $\Gamma=(V,E)$ so that it is the graph with these properties that has the smallest possible order, that is, if $\Gamma'\neq P_4$ is a connected graph with smallest degree $1$ and $3\leq N'<N$ vertices, then $\varepsilon(\Gamma')<\frac{1}{2}$. If we show that the existence of $\Gamma$ brings to a contradiction, then we are done.\newline
Fix $u\in V$ of degree $1$, and let $v\sim u$ be its only neighbour. Then, by letting $f(u):=1$ and $f:=0$ otherwise, \eqref{eq:important-equa} gives $\frac{1}{\deg v}\ge\frac14$, i.e., $\deg v\le 4$.\newline
Now, assume first that $\deg v=4$. Let $\hat{\Gamma}:=\Gamma\setminus\{u\}$, let $\hat{\varepsilon}:=\varepsilon(\hat{\Gamma})$ and let $f$ be a function on $V(\hat{\Gamma})=V\setminus\{u\}$ that is an eigenfunction for $(\id-\Delta(\hat{\Gamma}))^2$ with eigenvalue $\hat{\varepsilon}^2$. Then,
\begin{equation}\label{eq:hat2}
    \hat{\varepsilon}^2=\frac{\frac{1}{\deg v-1}\left(\sum_{x\in \mathcal{N}(v)\setminus \{u\}}f(x)\right)^2+\sum_{w\in V\setminus\{v,u\}}\frac{1}{\deg w}\left(\sum_{x\in \mathcal{N}(w)}f(x)\right)^2}{(\deg v-1)f(v)^2+\sum_{w\in V\setminus\{v,u\}}\deg w f(w)^2}.
\end{equation}
Moreover, since $\hat{\varepsilon}^2=(1-\hat{\lambda})^2$ for some eigenvalue $\hat{\lambda}$ of $\hat{\Gamma}$, from \eqref{eq:defL} it follows that
\begin{equation*}
    (1-\hat{\lambda})f(v)=\frac{1}{\deg v-1}\left(\sum_{x\in \mathcal{N}(v)\setminus \{u\}}f(x)\right),
\end{equation*}therefore
$$\hat{\varepsilon}\cdot (\deg v-1)\cdot |f(v)|=\left|\sum_{x\in \mathcal{N}(v)\setminus \{u\}}f(x)\right|$$ 
and similarly, for all $w\in V\setminus\{u,v\}$, $$\hat{\varepsilon} \cdot \deg w\cdot |f(w)|=\left|\sum_{x\in \mathcal{N}(w)}f(x)\right|.$$ 

In particular, without loss of generality, we may assume $$\hat{\varepsilon} \cdot (\deg v-1)\cdot f(v)=\sum_{x\in \mathcal{N}(v)\setminus \{u\}}f(x)$$ and  $$\hat{\varepsilon}\cdot \deg w\cdot f(w)=\sum_{x\in \mathcal{N}(w)}f(x),\quad \forall w\in V\setminus\{u,v\},$$
as the proof can be easily adapted otherwise.\newline 
Now, since we are assuming that $\varepsilon=\varepsilon(\Gamma) =\frac12$, we have that, independently of the value of $f(u)\in \mathbb{R}$ that we choose, \eqref{eq:important-equa} implies
$$\frac{\sum_{w\in V}\frac{1}{\deg w}\left(\sum_{x\in \mathcal{N}(w)}f(x)\right)^2}{\sum_{w\in V}\deg w f(w)^2}\ge\frac14,$$
or equivalently,
\begin{align*}
  & f(v)^2+\frac{1}{4}\cdot \left(\sum_{x\in \mathcal{N}(v)}f(x)\right)^2+\sum_{w\in V\setminus\{v,u\}}\frac{1}{\deg w}\left(\sum_{x\in \mathcal{N}(w)}f(x)\right)^2\\
  \ge~ & \frac14\sum_{w\in V}\deg w f(w)^2,  
\end{align*}
which by the above observations can also be rewritten as\begin{align*}
&f(v)^2+\frac{1}{4}(f(u)+\hat{\varepsilon} (\deg v-1)f(v))^2+\sum_{w\in V\setminus\{v,u\}}\frac{1}{\deg w}(\hat{\varepsilon} \cdot \deg w\cdot f(w))^2
\\\ge~ &\frac14 \left(f(u)^2+4\cdot f(v)^2
+
\sum_{w\in V\setminus\{v,u\}}\deg w f(w)^2\right).
\end{align*}We now consider two cases.
\begin{enumerate}[{Case }1.]
\item 
If $f(v)=0$, then the above inequality becomes $$\left(\hat{\varepsilon}^2-\frac14 \right)\cdot \left(\sum_{w\in V\setminus\{v,u\}}\deg w f(w)^2\right)\ge0.$$
Since $\hat{\varepsilon}^2\le\frac14$ by Theorem \ref{thm:part1} and $f(v)=0$ implies $f|_{V\setminus\{v,u\}}\ne0$, we deduce that $\hat{\varepsilon}^2=\frac14$. But this brings to a contradiction. In fact, by construction, it is clear that $\hat{\Gamma}$ is a connected graph on $N-1\geq 4$ vertices, since $v$ has degree $3$ in $\hat{\Gamma}$. Hence, if $\hat{\Gamma}$ has vertices of degree $1$, then by the assumption on the graphs with less than $N$ vertices, $\hat{\Gamma}$ must be $P_4$, but this is not possible because of the degree of $v$. Similarly, if $\hat{\Gamma}$ does not have vertices of degree $1$, then by Theorem \ref{thm:d2}, $\hat{\Gamma}$ is either a petal graph or a book graph, and in particular, since $v$ has degree $3$ in $\hat{\Gamma}$, then $\hat{\Gamma}$ must be the book graph on $8$ vertices. But in this case, one can directly check that $\varepsilon=\varepsilon(\Gamma)<\frac{1}{2}$, which is a contradiction.

\item If $f(v)\ne0$, then the above inequality becomes
$$\frac32\cdot \hat{\varepsilon}  f(u)f(v)+\frac94\cdot \hat{\varepsilon}^2 f(v)^2+\left(\hat{\varepsilon}^2-\frac14 \right)\cdot\left(\sum_{w\in V\setminus\{v,u\}}\deg w f(w)^2\right)\ge0.$$
If $\hat{\varepsilon}\ne0$, we can always derive a contradiction from the above inequality by choosing an appropriate value of $f(u)\in\mathbb{R}$. Therefore, $\hat{\varepsilon}=0$, implying that the inequality holds if and only if $f|_{V\setminus\{v,u\}}=0$. But in this case, \eqref{eq:hat2} is not satisfied, which is a contradiction.
\end{enumerate}
Hence, $\deg v=4$ always brings to a contradiction, implying that $\deg v\in\{2,3\}$, since we already know that $\deg v\leq 4$.\newline
Let now $\tilde{\Gamma}:=\Gamma\setminus\{u,v\}$, let $\tilde{\varepsilon}:=\varepsilon(\tilde{\Gamma})$ and let $f$ be a function on $V(\tilde{\Gamma})=V\setminus\{u,v\}$ that is an eigenfunction for $(\id-\Delta(\tilde{\Gamma}))^2$ with eigenvalue $\tilde{\varepsilon}^2$. Then,
\begin{small}
\begin{equation*}\label{eq:tilde2}
    \tilde{\varepsilon}^2=\frac{\sum_{w\in \mathcal{N}(v)\setminus \{u\}}\frac{1}{\deg w-1}\left(\sum_{x\in \mathcal{N}(w)\setminus \{v\}}f(x)\right)^2+\sum_{w\in V\setminus \mathcal{N}(v)}\frac{1}{\deg w}\left(\sum_{x\in \mathcal{N}(w)}f(x)\right)^2}{\sum_{w\in \mathcal{N}(v)\setminus \{u\}}(\deg w-1)f(w)^2+\sum_{w\in V\setminus \mathcal{N}(v)}\deg w f(w)^2}
\end{equation*}
\end{small}
and, similarly to the first part of the proof, without loss of generality we can assume that
$$\begin{cases}\sum_{x\in \mathcal{N}(w)\setminus \{v\}}f(x)=\tilde{\varepsilon}(\deg w-1)f(w),&\forall w\in \mathcal{N}(v)\setminus \{u\},\\
\sum_{x\in \mathcal{N}(w)}f(x)=\tilde{\varepsilon}\deg wf(w),&\forall w\in V\setminus \mathcal{N}(v).
\end{cases}$$
In fact, the case in which $\tilde{\varepsilon}$ is replaced by $-\tilde{\varepsilon}$ is similar.\newline
Since $\varepsilon=\varepsilon(\Gamma)=\frac{1}{2}$,  \eqref{eq:important-equa} holds for any value of $f(u)$ and $f(v)$. Therefore,
\begin{align*}
&f(v)^2+\frac{1}{\deg v}\left(f(u)+\sum_{w\in \mathcal{N}(v)\setminus \{u\}}f(w)\right)^2+\sum_{w\in V\setminus \mathcal{N}(v)}\frac{1}{\deg w}\left(\tilde{\varepsilon}\deg wf(w)\right)^2\\&    +\sum_{w\in \mathcal{N}(v)\setminus \{u\}}\frac{1}{\deg w}\left(f(v)+\tilde{\varepsilon}(\deg w-1)f(w)\right)^2
\\\ge~& \frac14\left(f(u)^2+
\deg vf(v)^2+\sum_{w\in V\setminus\{v,u\}}\deg w f(w)^2\right).
\end{align*}
This is equivalent to
\begin{align*}
&f(v)^2+\frac{1}{\deg v}\left(f(u)+\sum_{w\in \mathcal{N}(v)\setminus \{u\}}f(w)\right)^2+\left(\tilde{\varepsilon}^2-\frac14\right)\cdot \left(\sum_{w\in V\setminus \mathcal{N}(v)}\deg w f(w)^2\right)\\&    +\sum_{w\in \mathcal{N}(v)\setminus \{u\}}\frac{1}{\deg w}\left(f(v)+\tilde{\varepsilon}(\deg w-1)f(w)\right)^2
\\\ge~& \frac14\left(f(u)^2+
\deg vf(v)^2+\sum_{w\in \mathcal{N}(v)\setminus \{u\}}\deg w f(w)^2\right),
\end{align*}
which can be also rewritten as
\begin{align*}
&\left(1+\sum_{w\in \mathcal{N}(v)\setminus \{u\}}\frac{1}{\deg w}-\frac{\deg v}{4}\right)f(v)^2+\left(\frac{1}{\deg v}-\frac14\right)f(u)^2\\
& +\frac{2}{\deg v}\cdot f(u)\left(\sum_{w\in \mathcal{N}(v)\setminus \{u\}}f(w)\right)+f(v)\left(\sum_{w\in \mathcal{N}(v)\setminus \{u\}}\frac{2}{\deg w}\cdot \tilde{\varepsilon}(\deg w-1)f(w)\right)\\
&    +\frac{1}{\deg v}\left(\sum_{w\in \mathcal{N}(v)\setminus \{u\}}f(w)\right)^2
\\\ge~& \sum_{w\in \mathcal{N}(v)\setminus \{u\}}\left(\frac14\cdot \deg w f(w)^2-\frac{1}{\deg w}\left(\tilde{\varepsilon}(\deg w-1)f(w)\right)^2\right)\\
& +\left(\frac14-\tilde{\varepsilon}^2\right)\sum_{w\in V\setminus \mathcal{N}(v)}\deg w f(w)^2.
\end{align*}
Using the fact that $\deg v\in\{2,3\}$, we can finally complete the proof by considering the following three cases.

\begin{enumerate}[{Case }1:]
\item $\tilde{\varepsilon}^2\le\frac14$.

In this case, by letting \begin{equation*}
    f(u):=-\sum_{w\in \mathcal{N}(v)\setminus \{u\}}f(w)
\end{equation*}
and $f(v):=0$, the above inequality implies that
\begin{align}\label{0geq}
    0\ge~& \sum_{w\in \mathcal{N}(v)\setminus \{u\}}\left(\frac14\cdot \deg w f(w)^2-\frac{1}{\deg w}\left(\tilde{\varepsilon}(\deg w-1)f(w)\right)^2\right)\\\nonumber
& +\left(\frac14-\tilde{\varepsilon}^2\right)\cdot \left(\sum_{w\in V\setminus \mathcal{N}(v)}\deg w f(w)^2\right).
\end{align}
Since, for $w\in \mathcal{N}(v)\setminus\{u\}$,
\begin{align*}
&\frac14\cdot \deg w f(w)^2-\frac{1}{\deg w}\left(\tilde{\varepsilon}(\deg w-1)f(w)\right)^2
\\\ge~& \frac{f(w)^2}{4}\cdot \left(\deg w-\frac{(\deg w-1)^2}{\deg w}\right)     \\=~& \frac{f(w)^2}{4}\cdot \frac{2\deg w-1}{\deg w}\ge0,
\end{align*}the inequality \eqref{0geq} must be an equality, implying that $f|_{\mathcal{N}(v)\setminus \{u\}}=0$ and $\tilde{\varepsilon}^2=\frac14$. Thus, we have two cases:
\begin{enumerate}
    \item If $\tilde{\Gamma}$ is connected, then it is either a book graph or a petal graph.
    \item If $\tilde{\Gamma}$ is not connected, then it has two connected components and at least one of them is either a book graph or a petal graph.
\end{enumerate}
In both cases it is possible to find an eigenfunction $f$ for $(\id-\Delta(\tilde{\Gamma}))^2$ with eigenvalue $\tilde{\varepsilon}^2$, such that $f|_{\mathcal{N}(v)\setminus \{u\}}\neq 0$, which contradicts the above implication. 

\item  $\tilde{\varepsilon}^2>\frac14$ and $\deg v=2$.

In this case, 
$\tilde{\Gamma}=P_2$ and $\Gamma=P_4$, which is a contradiction since we are assuming that $\Gamma\neq P_4$.

\item $\tilde{\varepsilon}^2>\frac14$ and $\deg v=3$.

In this case, since we cannot have $\tilde{\Gamma}=P_2$ (as shown in the previous case), $\tilde{\Gamma}$ must have exactly two connected components, each of them is either $P_2$ or a single vertex. But in any of these cases, one can show that $\varepsilon=\varepsilon(\Gamma)<\frac12$, which is a contradiction. 
\end{enumerate}
This proves the claim. \end{proof}

\section*{Funding}
Raffaella Mulas was supported by the Max Planck Society's Minerva Grant.

\section*{Acknowledgments}
The authors are grateful to the anonymous referees for the comments and suggestions.

\bibliography{Gap1}
\end{document}